\documentclass[12pt]{article}

\usepackage[a4paper,width=155mm,top=25mm,bottom=25mm]{geometry}
\newcommand{\undertitle}[1]{}
\usepackage{amsthm}
\usepackage{titling}
\newtheorem{theorem}{Theorem}[section]
\newtheorem{lemma}[theorem]{Lemma}

\newtheorem{definition}[theorem]{Definition}

\newtheorem{remark}[theorem]{Remark}
\usepackage{mathtools}
\usepackage{xspace} 
\usepackage{mathrsfs}
\usepackage{esint}
\usepackage{algorithm2e}
\usepackage{algorithmic}
\usepackage{graphicx}
\usepackage{caption}
\usepackage{subcaption}
\usepackage[justification=centering]{caption} % Add this in the preamble
\usepackage{bm}

\usepackage{amsmath} 
\usepackage[utf8]{inputenc} % allow utf-8 input
\usepackage[T1]{fontenc}    % use 8-bit T1 fonts
\usepackage{url}            % simple URL typesetting
\usepackage{booktabs}       % professional-quality tables
\usepackage{amsfonts}       % blackboard math symbols
\usepackage{nicefrac}       % compact symbols for 1/2, etc.
\usepackage[numbers]{natbib}
\usepackage{doi}
\usepackage{mathrsfs}

\def \Mvec{\mathbf{M}}

\def\Hvec{\mathbf{H}}

\graphicspath{{Pictures/}}

\usepackage{hyperref}       % hyperlinks
\usepackage{cleveref}

\title{Existence of homeomorphic minimizers via mappings of finite distortion in compressible magnetoelasticity}

\author{
Shilpa Dutta\thanks{Institute of Mathematics, 
University of W\"urzburg, Germany,\newline 
\href{mailto:shilpa.dutta@uni-wuerzburg.de}{\texttt{shilpa.dutta@uni-wuerzburg.de}}} ,\quad
Anja Schlömerkemper\thanks{Institute of Mathematics,
University of W\"urzburg, Germany,\newline
\href{mailto:anja.schloemerkemper@uni-wuerzburg.de}{\texttt{anja.schloemerkemper@uni-wuerzburg.de}}}
}

\begin{document}
\maketitle

\begin{abstract}
We establish the existence of an energy minimizer for a variational model of compressible magnetoelastic solids. The analysis is carried out in a new admissible class of deformations consisting of mappings of finite distortion, which extends previously available existence frameworks. A key ingredient is a compactness result under the critical integrability assumption on the outer distortion coefficient, which significantly weakens the regularity requirements imposed in earlier works.

To obtain this result, we prove a diameter estimate for finite-distortion mappings satisfying the Ciarlet-Ne\v{c}as condition and derive an open mapping theorem under the optimal integrability assumption, that is, the outer distortion is in $L^{n-1}$. This provides a partial positive result in the direction of the Iwaniec-\v{S}ver\'ak conjecture and implies that admissible deformations are homeomorphisms.

These topological and compactness properties allow us to apply the direct method of the calculus of variations and establish the existence of minimizers for compressible magnetoelastic solids within the admissible class of deformations consisting of mappings of finite distortions.
\end{abstract}
\\[0.5em]
\textbf{\textit{MSCcodes:} Primary 49J27; Secondary 74F15}

\section{Introduction}\label{sec:1}
%Structure of the new version of the introduction: Magnetoelasticity and the variational problem; Existing existence results and their limitations; The finite-distortion framework and the Iwaniec-Šverák motivation; Main contributions of the paper; Outline of the article.

Magnetoelastic materials are ferromagnetic solids whose mechanical deformation and magnetic response are strongly coupled. Mechanical stresses may change the magnetic configuration of the material (inverse magnetostriction), while changes in magnetization may induce mechanical deformation (magnetostriction). Because of this two-way coupling, magnetoelastic materials play an important role in a variety of applications such as bioelectronic devices, sensors, actuators and shape memory alloys \cite{Nature, brown1966magnetoelastic}. In many applications, the material is subjected to an external magnetic field.  Ferromagnetic solids show a spontaneous magnetization $\Mvec$, whose magnitude $|\Mvec|$ depends on temperature.\newline\\
In this work, we consider a variational model for compressible magnetoelastic solids in the mixed Eulerian-Lagrangian formulation, which mainly follows \cite{brown1966magnetoelastic, hubert2008magnetic}.\newline 
Let $\Omega\subset\mathbb R^3$ denote the reference configuration of a magnetoelastic solid and let $\bm{y}: \Omega \rightarrow \mathbb R^3$ represent its deformation.
%such that $\bm{y}\left(X\right)=x$. 
The magnetization is defined on the deformed configuration, $\Mvec:  \bm{y}\left(\Omega\right)\rightarrow \mathbb R^3$. Following \cite{brown1966magnetoelastic, hubert2008magnetic, James, MK}, we consider the total magnetoelastic energy 
\begin{eqnarray}
\mathscr{E}_{mag} \left(\bm{y}, \Mvec\right)
&=& \int_\Omega \mathcal{W} \left(\nabla \bm{y}\left(X\right) , \Mvec\left(\bm{y}\left(X\right)\right)\right) dX + \alpha \int_{\bm{y}\left(\Omega\right)} |\nabla \Mvec\left(x\right)|^2dx \nonumber\\
&&+ \dfrac{\mu_0}{2} \int_{\mathbb R^3} |\Hvec_{\Mvec}\left(x\right) |^2dx.
\label{eq:en_mag}
\end{eqnarray}
\noindent The first term in the energy \eqref{eq:en_mag} is the magnetostrictive energy and describes the coupling between deformation and magnetization. The energy density $\mathcal{W}: \mathbb R^{3\times 3} \times \mathbb R^3 \rightarrow \left[0, \infty\right)$ accounts for material-dependent effects such as magnetocrystalline anisotropy and is generally nonconvex and nonlinear. In particular, it commonly contains terms that are quadratic in the magnetization components \cite{hubert2008magnetic}.
We assume the standard frame-indifference, i.e.,
\begin{align}
\mathcal{W}\left(\mathcal{R}\bm{F},\mathcal{R}\bm{\lambda} \right) = \mathcal{W}\left( \bm{F}, \bm{\lambda}\right) \text{ for any } \bm{F}\in \mathbb R^{3\times 3},~\bm{\lambda}\in\mathbb R^3, \mathcal{R} \in  SO(3).\nonumber
\end{align}
The second term in \eqref{eq:en_mag} is the exchange energy, which penalizes spatial variations of the magnetization and therefore promotes magnetic alignment. The material constant $\alpha>0$ is generally temperature dependent and known as the exchange constant \cite{hubert2008magnetic}. The third term is the magnetostatic or stray-field energy associated with the magnetic field generated by the magnetization itself. The corresponding magnetic field $\Hvec_{\Mvec}:\mathbb R^3\rightarrow\mathbb R^3$ is determined by the stationary Maxwell equations (see e.g.,\ \cite{James, MK})
\begin{gather}
\begin{aligned}
\textbf{div}\left(\chi_{\bm{y}\left(\Omega\right)}\Mvec + \mu_0\Hvec_{\Mvec}\right)&=0 \mbox{ in } \mathbb R^3,\\
\textbf{curl}~\Hvec_{\Mvec}&=0 \mbox{ in } \mathbb R^3.
\end{aligned}\label{eq:maxwell}
\end{gather}
Here $\chi_{\bm{y}\left(\Omega\right)}:\mathbb R^3\rightarrow [0, 1]$ is the characteristic function of $\bm{y}\left(\Omega\right)$ and the constant $\mu_0>0$ is the permeability of the vacuum \cite{hubert2008magnetic}.\\
The model is supplemented by standard assumptions from nonlinear elasticity and micromagnetics. In particular, admissible deformations are orientation preserving,
 $\det(\nabla \bm{y})>0$ almost everywhere in $\Omega$, which excludes interpenetration of matter. Moreover, following the compressible magnetoelastic theory of Brown and James-Kinderlehrer \cite{brown1966magnetoelastic,James,Luskin}, we impose the magnetic saturation constraint
\begin{eqnarray}\label{eq:sat1}
|\Mvec\left(\bm{y}\left(X\right)\right)|\det(\nabla \bm{y}\left(X\right)) = 1 \mbox{ a.e. in } \Omega.
\end{eqnarray}
This condition reflects the dependence of the magnetic moment density on local volume changes. In the incompressible case, for which $\det(\nabla \bm{y}\left(X\right))=1$, it reduces to the classical unit-length constraint
\begin{eqnarray}\label{eq:sat2}
|\Mvec\left(x\right)|=1 \mbox{ a.e. in } \bm{y}\left(\Omega\right).
\end{eqnarray}
\subsection{The state of the art}
\label{ch:literature review}
The existence theory for magnetoelastic energy functionals originates from the program of extending the direct methods of nonlinear elasticity \cite{ball1976convexity,ball1981null} to coupled magneto-mechanical models. Early contributions include the works of Rogers and coauthors \cite{rogers1986steady,rogers1988nonlocal,rogers1993existence}, where magnetic energies were reformulated in Lagrangian coordinates and strong constitutive assumptions were imposed on the elastic energy density.

In a mixed Eulerian--Lagrangian framework, existence results were first obtained by DeSimone and coauthors \cite{Desimmone} in two space dimensions. Later, James and Kinderlehrer \cite{Luskin} established the existence of minimizers in three dimensions under the physically relevant compressibility constraint \eqref{eq:sat1}. Their admissible class requires deformations in \(W^{2,2}\), and the higher regularity is used to preserve injectivity through arguments based on Banach's indicatrix.

For incompressible magnetoelastic solids, the authors in \cite{MK} proved existence of minimizers in the class \(W^{1,p}\), \(p>3\). The continuity obtained from Morrey's embedding simplifies the compactness analysis, while the saturation condition \eqref{eq:sat2} yields strong convergence properties for the magnetization-deformation composition. However, the compressible constraint \eqref{eq:sat1} is generally regarded as the more realistic physical assumption \cite{brown1966magnetoelastic}.

Recent works have considerably broadened the admissible deformation classes. In \cite{Bresciani21} and \cite{Barchiesi}, existence theories were established for subcritical Sobolev regular deformations using topological degree techniques. More recently, \cite{Bresciani22} extended these ideas to compressible magnetoelastic materials under the saturation constraint \eqref{eq:sat1}.
These approaches rely on suitable coercivity and polyconvexity assumptions and provide a powerful framework for treating magnetoelastic models beyond the classical Sobolev homeomorphism setting.

Mappings of finite distortion provide a natural framework for studying deformations with weak regularity while retaining quantitative control of local stretching and volume change. The theory has been extensively developed in geometric function theory and nonlinear elasticity; see, for example, \cite{Hencl-Koskela}. Within elasticity, distortion integrability is closely related to continuity, openness, invertibility and compactness properties of admissible deformations. Recent developments in this direction include \cite{dolevzalova2025global,dolevzalova2025inv}.

\subsection{Motivation from finite distortion theory} 

Mappings of finite distortion form a considerably larger deformation class than the Sobolev homeomorphism classes commonly used in nonlinear elasticity, while still retaining geometric information through the distortion coefficients. A central analytical ingredient of the present work is an open mapping theorem for mappings of finite distortion. The result is closely related to a well-known problem in geometric function theory, namely the Iwaniec-\v{S}ver\'ak conjecture \cite{Iwaniec_Sverak}. In its classical form, the conjecture asks whether every nonconstant mapping of finite distortion 
$$\bm f\in W^{1,p}\left(\Omega;\mathbb R^n\right), \qquad p\ge n, $$
with outer distortion coefficient 
$$K^O_{\bm f}\in L^{n-1}(\Omega) $$
is necessarily open. For $n=2$, the conjecture has been confirmed (see \cite{Hencl-Koskela}), while for dimensions \(n\ge3\) the problem remains largely open. Several partial results are available under additional assumptions on the distortion or the mapping itself or imposing some condition on the mapping close to the boundary of the domain \cite{Hencl-Koskela}.  In the present work, we contribute to this direction by proving openness under the additional assumption that the mapping satisfies the Ciarlet-Ne\v{c}as condition. The obtained result applies precisely in the borderline integrability regime \(K^O_{\bm f}\in L^{n-1}\). Beyond its independent interest in finite distortion theory, this openness result provides the key topological ingredient required for the magnetoelastic model prescribed in mixed Euler-Lagrange setting as considered here. It allows us to establish that admissible deformations are homeomorphisms and thereby obtain a natural admissible framework for compressible magnetoelastic solids.

While mappings of finite distortion provide a flexible framework for relatively weak regular deformations, they do not in general guarantee topological properties such as openness, injectivity, or global invertibility. Recovering these properties under sharp integrability assumptions on the distortion is therefore a central analytical challenge. The openness theorem established in this work addresses this issue and provides the topological foundation required for the subsequent existence theory.
 
\subsection{Main contributions}

The starting point of the present work is the observation that existence theories for magnetoelastic solids remain strongly influenced by the regularity assumptions imposed on admissible deformations. Our goal is to replace these higher-order regularity assumptions by geometric conditions expressed in terms of mappings of finite distortion.

The first contribution (Lemma \ref{lem2}) is a diameter estimate for inverse images of balls under finite distortion mappings satisfying the Ciarlet-Ne\v{c}as condition. Building on this estimate, we prove an open mapping theorem (\Cref{thm2}) under the critical integrability assumption
\begin{eqnarray*}
K^O_{\bm y}\in L^{n-1}.
\end{eqnarray*}
This result may be viewed as a partial positive result related to the Iwaniec-\v{S}ver\'ak conjecture on openness of finite distortion mappings \cite{Iwaniec_Sverak}. In particular, under the additional Ciarlet-Ne\v{c}as condition we establish openness in the borderline integrability regime.

These topological properties imply that all admissible deformations in our framework are homeomorphisms. This allows us to define a new admissible class for compressible magnetoelastic solids consisting of homeomorphic mappings of finite distortion and magnetizations satisfying the compressibility constraint.

Finally, we establish compactness of this admissible class and prove lower semicontinuity of the magnetoelastic energy. As a consequence, the direct method of the calculus of variations yields the existence of minimizers for compressible magnetoelastic solids in the proposed finite distortion setting (\Cref{thm1}).

\subsection{Organization of the paper} 

The remainder of the paper closely follows Chapter 3 of the first author's PhD thesis \cite{Dutta-thesis} and is organized as follows. 

In Section~\ref{preliminary notions}, we introduce the class of mappings of finite distortion and collect the analytical tools required in Section~\ref{Key lemmas and the main theorem}. We also formulate the assumptions imposed on the magnetostrictive energy density and discuss the coercivity framework underlying the variational problem. 

Section~\ref{Key lemmas and the main theorem} is devoted to the fundamental geometric and topological properties of admissible deformations. We first establish a diameter estimate for inverse images of balls under mappings of finite distortion satisfying the Ciarlet-Ne\v{c}as condition. This estimate is then used to prove an open mapping theorem under the critical integrability assumption on the outer distortion coefficient. As a consequence, the admissible deformations considered in this work possess the topological properties needed for the variational analysis. 

In Section~\ref{Existence results and proof}, we introduce the admissible class for compressible magnetoelastic solids and prove that every admissible deformation is a homeomorphism. We then establish a compactness theorem for energy-bounded sequences and prove strong convergence properties for the magnetization in composition for a homeomorphic deformation. 

The final part of Section~\ref{Existence results and proof} is devoted to the lower semicontinuity analysis of the magnetoelastic energy. Combining the compactness and lower semicontinuity results, we apply the direct method of the calculus of variations to obtain the existence of minimizers for the compressible magnetoelastic energy functional.

\section{Preliminary notions}
\label{preliminary notions}
In this section, we recall the basic notions and assumptions that will be used in the sequel. We first introduce mappings of finite distortion and the associated distortion coefficients, and then state the hypotheses imposed on the magnetostrictive energy density.
\subsection{Finite distortion}
\label{finite distortion}
\begin{definition}(Finite distortion \cite[page 14]{Hencl-Koskela}) Let $\Omega\subset\mathbb R^n$ be an open connected set for some $n\ge 2$. Then a mapping $\bm{f}:\Omega\rightarrow\mathbb R^n$ is called a mapping of finite distortion if $\bm{f}\in W^{1,1}_{loc}\left(\Omega; \mathbb R^n\right)$, the Jacobian $J_{\bm{f}}\in L^1_{loc}\left(\Omega\right)$ and there exists a function $K:\Omega\rightarrow [1, \infty]$ with $K\left(x\right)<\infty$ a.e.\ such that
\begin{eqnarray}
|\nabla\bm{f}\left(x\right)|^n\le K\left(x\right)J_{\bm{f}}\left(x\right) \mbox{ for almost all } x\in\Omega.\label{eq:finite_dis}
\end{eqnarray}
\end{definition}
For mappings of finite distortion $\bm{f}:\Omega\rightarrow\mathbb R^n$~$\left(n\ge 2\right)$, the outer distortion $K^O_{\bm{f}}:\Omega\rightarrow\mathbb R$ is defined as 
\begin{eqnarray}\label{eq:outer_distortion}
K^{O}_{\bm{f}}\left(x\right) =
\begin{cases}
\dfrac{|\nabla \bm{f}\left(x\right)|^n}{J_{\bm{f}}\left(x\right)} & \text{if $J_{\bm{f}}\left(x\right)> 0$},\\
1 & \text{if $J_{\bm{f}}\left(x\right)=0$}.
\end{cases}
\end{eqnarray}
In addition, we introduce the inner distortion coefficient $K^I_{\bm{f}}:\Omega\rightarrow\mathbb R$, which measures the distortion in terms of the adjugate matrix and is given by is defined by
\begin{eqnarray}\label{eq:inner_distortion}
K^{I}_{\bm{f}}\left(x\right) =
\begin{cases}
\dfrac{|\mbox{adj}\left(\nabla \bm{f}\left(x\right)\right)|^n}{J_{\bm{f}}\left(x\right)^{n-1}} & \text{if $J_{\bm{f}}\left(x\right)> 0$},\\
1 & \text{if $J_{\bm{f}}\left(x\right)=0$}.
\end{cases}
\end{eqnarray}
The outer and inner distortion coefficients $K^{O}$ and $K^{I}$ are connected through the following pointwise bounds (cf.\ \cite[Lemma 3.3]{yan2000weak}):
\begin{eqnarray}
\left(K^I_{\bm{f}}\right)^{\frac{1}{n-1}}\left(x\right)\le K^O_{\bm{f}}\left(x\right)\le \left(K^I_{\bm{f}}\right)^{{n-1}}.
\end{eqnarray}
\subsection{Coercivity condition on the magnetoelastic energy functional}
\label{sec:coercivity}
Throughout the paper, the reference configuration $\Omega\subset\mathbb R^3$ is assumed to be open and connected. The magnetostrictive energy density $\mathcal{W}:\mathbb R^{3\times 3}\times\mathbb R^3\rightarrow\mathbb R$ is required to satisfy the following assumptions.\newline
\\
\textbf{(i) Polyconvexity of $\mathcal{W}$:}  There exists a continuous function $\overline{\mathcal{W}}:\mathbb R^{3\times 3}\times\mathbb R^{3\times 3}\times \left(0, \infty\right)\times\mathbb R^3\rightarrow\mathbb R$ such that the function $\overline{\mathcal{W}}\left(\cdot,~\cdot,~\cdot, \bm{\lambda}\right)$ is convex for all $\bm{\lambda}\in\mathbb R^3$ and the following identity holds
\begin{eqnarray}
\mathcal{W}\left(\nabla\bm{y},\bm{\lambda}\right)=\overline{\mathcal{W}}\left(\nabla\bm{y}, \mathbf{Cof}\left(\nabla\bm{y}\right), \det(\nabla\bm{y}), \bm{\lambda}\right)\label{eq:poly}
\end{eqnarray}
for all $\nabla\bm{y}\in\mathbb{R}^{3\times 3}_+$, $\bm{\lambda}\in\mathbb R^3$.\newline
\\
\textbf{(ii) Coerciveness of $\mathcal{W}$:} There exist constants $\alpha_1, \alpha_2>0, p>3, q\ge 2, s>1$ such that
\begin{eqnarray}
\mathcal{W}\left(\nabla\bm{y}, \bm{\lambda}\right)\ge \alpha_1\left(|\nabla\bm{y}|^p + |K^{O}_{\bm{y}}|^q + \left(\det(\nabla\bm{y})\right)^{-s}\right) + \alpha_2\label{eq:coer}
\end{eqnarray}
for all $\nabla\bm{y}\in\mathbb{R}^{3\times 3}_+$, $\bm{\lambda}\in\mathbb R^3$.\newline\\
\textbf{(iii) Behavior of $\mathcal{W}$ under extreme compression, i.e., when $\det(\nabla\bm{y})\rightarrow 0^{+}$:} for all $\bm{\lambda}\in\mathbb R^3$,
\begin{eqnarray}
\lim_{\det(\nabla\bm{y})\rightarrow 0^{+}}\overline{\mathcal{W}}\left(\nabla\bm{y}, \bm{\lambda}\right)\rightarrow +\infty.\label{eq:beh}
\end{eqnarray}
We furthermore assume that the energy density is frame-indifferent, that is, 
\begin{eqnarray}
\mathcal{W}\left(R\nabla\bm{y}, R\bm{\lambda}\right)=\mathcal{W}\left(\nabla\bm{y},\bm{\lambda}\right)
\end{eqnarray}
for all $\nabla\bm{y}\in\mathbb{R}^{3\times 3}_+$, for all $\bm{\lambda}\in\mathbb R^3$, and $R\in SO(3)$.\newline
\section{Key lemmas and the main theorem}\label{Key lemmas and the main theorem}
In this section, we establish the main geometric and topological properties of admissible deformations. We begin by recalling several auxiliary results that will be used in the proof of the openness theorem. The central ingredient is a diameter estimate for inverse images of balls under mappings of finite distortion. The argument is inspired by the estimate in \cite[Lemma 2.1]{Onninen}, which extends a result from \cite[Lemma 1.2]{Hencl-Koskela} to three dimensions. In contrast to these earlier works, we do not assume that the deformation is a homeomorphism a priori. Instead, the estimate is derived from the finite distortion framework together with the integrability assumption on the outer distortion coefficient.
\begin{lemma}(\cite[page 143]{Evans-Gariepy})\label{lem1'}
Let $\Omega\subset\mathbb R^n$ be bounded, open and connected, and $\bm{y}\in W^{1,p}_{loc}\left(\Omega; \mathbb R^n\right)$, where $n<p<\infty$. Then there exists a constant $C_p\left(n\right)$ depending only on $n$ and $p$, such that
\begin{align}
|\bm{y}\left(x_1\right) - \bm{y}\left(x_2\right)|\le C_{p}\left(n\right)s^{1-\frac{n}{p}}\left(\int_{B_s}|\nabla \bm{y}|^pdx\right)^{\frac{1}{p}}\label{eq:lem1}
\end{align}
for almost every $x_1, x_2\in B_s = B\left(z, s\right)\subset \Omega$.
\end{lemma}
\begin{remark}\label{re-lem1}
The previous result can also be applied on $\left(n-1\right)$-dimensional slices. The reduction is obtained as follows.  
 Let $x\in\Omega\subset\mathbb R^n$. For each $\theta$, we define a set
$$\Omega_{\theta}=\{x'\in\mathbb R^{n-1}: \left(x', \theta\right)\in\Omega\}.$$
Since $\bm{y}\in W^{1,p}_{loc}\left(\Omega\right)$, due to Fubini-Tonelli's theorem (see e.g.\ \cite[Theorem 1.121, Theorem 1.124]{Fonseca-Leoni}), $\bm{y}\left(\cdot, \theta\right)\in W^{1,p}_{loc}\left(\Omega_{\theta}\right)$ for almost every $\theta\in\mathbb R$. Therefore, for an $\left(n-1\right)$-dimensional ball, Lemma \ref{lem1'} is understood as
\begin{align}
|\bm{y}\left(x', \theta\right)-\bm{y}\left(z', \theta\right)|\le C_p\left(n\right)s^{1 - \frac{n-1}{p}}\left(\int_{B^{n-1}_s}|\nabla_{x'} \bm{y}\left(\cdot, \theta\right)|^pdx'\right )^{\frac{1}{p}}
\end{align}
for almost every $\theta$.
\end{remark}
We next derive a diameter estimate that will play a crucial role in the proof of the openness theorem. The argument relies on Morrey's inequality together with the finite distortion structure of the deformation. As a crucial preliminary step, we show that sufficiently small balls around suitable image points are contained in the deformed configuration $\bm{y}\left(\Omega\right)$. 
\begin{lemma}\label{lem2}
Let $\Omega\subset\mathbb R^n$ be bounded, open, and connected. Suppose that $\bm{y}\in W^{1,p}\left(\Omega; \mathbb R^n\right)$, where $p\ge n$ for $n\ge 3$, is a continuous mapping of finite distortion that satisfies the following properties
\begin{enumerate}
    \item $K^O_{\bm{y}}\in L^{n-1}\left(\Omega\right)$,
    \item $\det(\nabla\bm{y})>0$ a.e.\ in $\Omega$,
    \item $\bm{y}$ satisfies the Ciarlet-Ne\v{c}as condition: $\int_{\Omega}\det(\nabla\bm{y}\left(X\right))dX\le\mu\left(\bm{y}\left(\Omega\right)\right)$.
    \end{enumerate} Then there is a constant $C_p\left(n\right)$ such that
\begin{eqnarray}
\operatorname{diam}\left(\bm{y}^{-1}\left(B_r\right)\right)\le C_p\left(n\right)r^{1-n}[\mu\left(\bm{y}^{-1}\left(B_{2r}\right)\right)]^{\frac{p-n+1}{p}}\left(\int_{\bm{y}^{-1}\left(B_{2r}\right)}|\nabla \bm{y}|^p\right)^{\frac{n-1}{p}}\label{eq:lem2}
\end{eqnarray}
for all balls $B_r = B\left(\bm{y}\left(X\right), r\right)$ such that $B_{3r} = B\left(\bm{y}\left(X\right), 3r\right)\subset \bm{y}\left(\Omega\right)$ and $X\in\Omega\setminus\left(\mathcal{K}\cup \mathcal{S}\right)$, where 
\begin{eqnarray*}
    \mathcal{K} &=& \{X\in\Omega: \bm{y} \mbox{ is not differentiable at } X\} \\ 
    \mathcal{S} &=& \{X\in\Omega: \bm{y} \mbox{ is differentiable at } X \mbox{ and } \det(\nabla \bm{y}\left(X\right)) = 0\}.
\end{eqnarray*}
\end{lemma}
\begin{proof}
The proof is devided into two steps.\newline
\textbf{Step I.}\newline
Since $\bm{y}$ is continuous and satisfies $\det(\nabla \bm{y}\left(X\right))>0$ almost everywhere, we first verify that image points corresponding to \(X\in \Omega\setminus(K\cup S)\) possess neighborhoods contained in \(y(\Omega)\). More precisely, we assert that for a continuous function $\bm{y}: \Omega\rightarrow\mathbb R^n$, which satisfies the Lusin ($\it{N}$) condition and $\det(\nabla \bm{y}\left(X\right))>0$ a.e. in $\Omega$, there exists $r>0$ such that $B\left(\bm{y}\left(X\right), r\right)\subset \bm{y}\left(\Omega\right)$ for any $X\in \Omega\setminus\left(\mathcal{K}\cup \mathcal{S}\right)$. Now let $X\in\Omega\setminus\left(\mathcal{K}\cup\mathcal{S}\right)$. Due to  \cite[Lemma 3.5.]{fonseca1995local}, there exists $r_X>0$ such that $B\left(X, r_X\right)\subset\subset\Omega$ and $\mbox{deg}\left(\bm{y}\left(X\right), \bm{y}, B\left(X, r_X\right)\right)>0$. We take the connected component of $\mathbb R^n\setminus \bm{y}\left(\partial B\left(X, r_X\right)\right)$, which contains $\bm{y}\left(X\right)$ and denote it by $C$. Using the definition of outer distortion $K^O_{\bm{y}}$ cf.\ \eqref{eq:outer_distortion}, the assumption $K^O_{\bm{y}}\ge 1$, and $\det(\nabla\bm{y})>0$ almost everywhere, we obtain
\begin{eqnarray}
\frac{|\nabla \bm{y}|^n}{\log(e+|\nabla \bm{y}|)}&\le& \frac{K^O_{\bm{y}}\det(\nabla \bm{y})}{\log(e+\left(K^O_{\bm{y}}\right)^{\frac{1}{n}}\det(\nabla \bm{y})^{\frac{1}{n}})}\nonumber\\
&\le& \frac{K^O_{\bm{y}}\det(\nabla \bm{y})}{\log(e+\det(\nabla \bm{y})^{\frac{1}{n}})}\nonumber\\
&\le& n\frac{K^O_{\bm{y}}\det(\nabla \bm{y})}{\log(e+\det(\nabla \bm{y}))}\nonumber\\
&\le& n\frac{K^O_{\bm{y}}\left(e-1+\det(\nabla \bm{y})\right)}{\log\left[1+\left(e-1+\det(\nabla \bm{y})\right)\right]}.\label{eq:log1}
\end{eqnarray}
To control the logarithmic term appearing above, we take the term $\frac{e-1+\det(\nabla \bm{y})}{\log\left[1+\left(e-1+\det(\nabla \bm{y})\right)\right]}$ and let $\zeta=e-1+\det(\nabla \bm{y})$. Further we consider that
\begin{eqnarray}
f\left(\zeta\right)= \log\left(1+\zeta\right)-\frac{\zeta}{1+\zeta}.
\end{eqnarray}
A direct computation shows that $f'\left(\zeta\right)=\frac{\zeta}{\left(1+\zeta\right)^2}>0$, when $\zeta>0$ and $f'\left(\zeta\right)=0$, when $\zeta=0$, $f\left(\zeta\right)$ is strictly increasing for $\zeta\in\left[0, \infty\right)$. Thus we have
\begin{eqnarray}
f\left(\zeta\right)>f\left(0\right) \mbox{ for } \zeta>0,
\end{eqnarray}
which yields
\begin{eqnarray}
\log\left[1+\left(e-1+\det(\nabla \bm{y})\right)\right]>\frac{e-1+\det(\nabla \bm{y})}{1+\left(e-1+\det(\nabla \bm{y})\right)}.\label{eq:log2''}
\end{eqnarray}
Next we insert \eqref{eq:log2''} into \eqref{eq:log1} and obtain
\begin{eqnarray}
\int_{\Omega}\frac{|\nabla \bm{y}|^n}{\log(e+|\nabla \bm{y}|)}< n\int_{\Omega}K^O_{\bm{y}}[1+\left(e-1+\det(\nabla \bm{y})\right)]<\infty,\label{eq:log3}
\end{eqnarray}
where the last inequality follows due to the fact that $K^O_{\bm{y}}\in L^{n-1}\left(\Omega\right)$ and by using the definition of $K^O_{\bm{y}}$. Since $|\nabla \bm{y}|\in L^n\log^{-1}L\left(\Omega\right)$, we obtain $\mbox{deg}\left(C, \bm{y}, B\left(X, r_X\right)\right)>0$ by \cite[Theorem 3.16]{Hencl-Koskela}. Hence, by the property of topological degree (e.g.\ \cite[page 48]{Hencl-Koskela}), $C\subset \bm{y}\left(B\left(X, r_X\right)\right)$. Since $\bm{y}\left(\partial B\left(X, r_X\right)\right)$ is a compact subset in $\mathbb R^n$, its complement is open. Consequently, $\mathbb R^n\setminus \bm{y}\left(\partial B\left(X, r_X\right)\right)$ is open, and the assertion follows.
\medskip  
\\
\noindent Having established the geometric inclusion above, we now turn to the proof of inequality \eqref{eq:lem2}.\newline\\
\textbf{Step II.}\newline
Fix $X\in\Omega\setminus\left(\mathcal{K}\cup \mathcal{S}\right)$ and consider $$d=\operatorname{diam}\left(\bm{y}^{-1}\left(B_r\right)\right).$$ Without loss of generality, we assume that $\overline{\bm{y}^{-1}\left(B_r\right)}$ contains the origin and the point $\left(0, 0, \ldots, 0, d\right)$. For any $\theta\in \left(0, d\right)$, we choose a point $z(\theta)\in \bm{y}^{-1}\left(B_r\right)$ such that the $n$-th coordinate of $z(\theta)$ lies between $0$ to $d$, cf. \Cref{Fig:Ball}, and we denote
\begin{align*}
    \mathcal{Q}_{\theta} = \{s\left(\theta\right)\in\mathbb R: \;& B^{n-1}\left(z\left(\theta\right), s\left(\theta\right)\right)\times\{\theta\}\subset \bm{y}^{-1}\left(B_{2r}\right) \mbox{ and }\\
& \partial B^{n-1}\left(z\left(\theta\right), s\left(\theta\right)\right)\cap\partial\bm{y}^{-1}\left(B_{2r}\right)\ne\phi\},
\end{align*}
where $B^{n-1}\left(z\left(\theta\right), s\left(\theta\right)\right)$ stands for the $\left(n-1\right)$-dimensional ball of radius $s\left(\theta\right)>0$ with center at $z\left(\theta\right)$. Such a radius $s\left(\theta\right)$ can be chosen as the smallest distance from $z\left(\theta\right)$ to the boundary of $\bm{y}^{-1}\left(B_{2r}\right)$ for each $\theta$, see \Cref{Fig:Ball}.\newline
\begin{figure}
    \centering
    \includegraphics[width=0.55\linewidth]{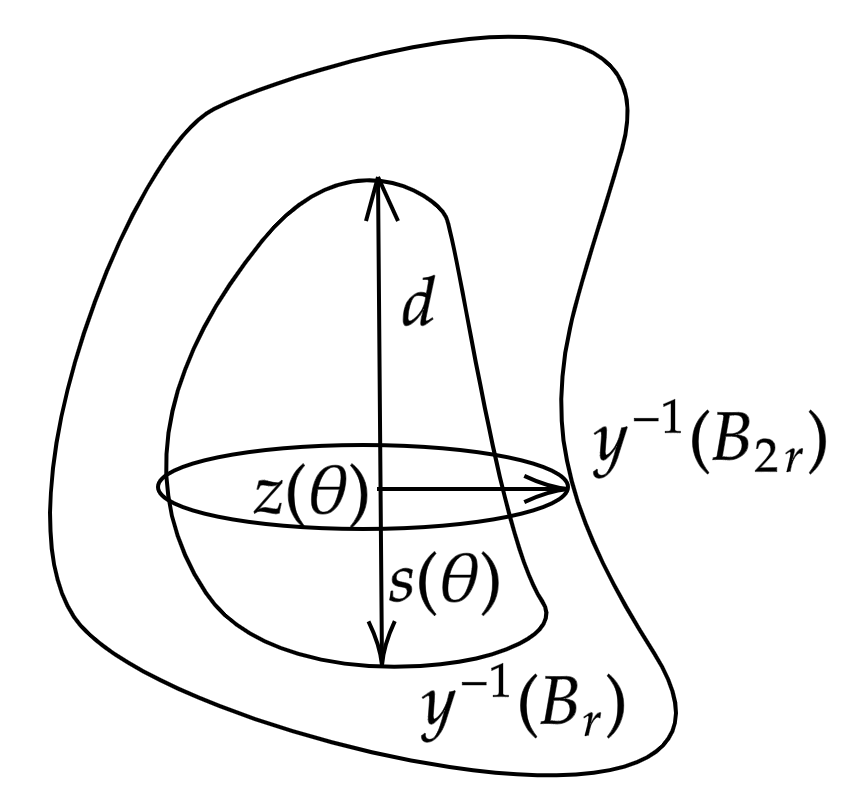}
    \caption{An example of the geometric setting for a $\theta \in (0,d)$  showing an $\left(n-1\right)$-dimensional ball $B^{n-1}\left(z\left(\theta\right),s\left(\theta\right)\right)$ such that the boundary of the ball $B^{n-1}\left(z\left(\theta\right),s\left(\theta\right)\right)$ touches the boundary of the outer ball $\bm{y}^{-1}\left(B_{2r}\right)$.}
    \label{Fig:Ball}
\end{figure}
\newline
As $p> n-1$, in view of Remark \ref{re-lem1}, for almost every $\theta$, we have
\begin{align}
|\bm{y}\left(x', \theta\right)-\bm{y}\left(y', \theta\right)|\le C_p\left(n\right)s\left(\theta\right)^{1 - \frac{n-1}{p}}\left(\int_{B^{n-1}\left(z\left(\theta\right),\ s\left(\theta\right)\right)}|\nabla_{x'} \bm{y}\left(\cdot, \theta\right)|^pdx'\right )^{\frac{1}{p}}\label{eq:le6}
\end{align}
for any $\left(x', \theta\right), \left(y', \theta\right)\in B^{n-1}\left(z\left(\theta\right),\ s\left(\theta\right)\right)$.\newline\\
Applying the same reasoning as in \cite[Lemma 2.19]{Hencl-Koskela}, we obtain
\begin{eqnarray}
\mbox{osc}_{B^{n-1}\left(z\left(\theta\right),~s\left(\theta\right)\right)}\left(\bm{y}\right)\le C_p\left(n\right)s\left(\theta\right)^{1 - \frac{n-1}{p}}\left(\int_{B^{n-1}\left(z\left(\theta\right),~s\left(\theta\right)\right)}|\nabla_{x'} \bm{y}\left(\cdot, \theta\right)|^pdx'\right )^{\frac{1}{p}}.
\end{eqnarray}
The definition of $\mathcal{Q}_{\theta}$ yields
%Since $\nabla y\in L^{n-1}$ and $\mbox{adj}\left(\nabla \bm{y}\right)\in L^{\frac{n}{n-1}}$, the mapping of finite distortion $\bm{y}$ is weakly monotone \cite{iwaniec2001mappings}. Because of continuity, $\bm{y}$ is monotone (see e.g.\ \cite[Theorem 2.24]{Hencl-Koskela}). Therefore
\begin{eqnarray}
r\le C_p\left(n\right)s\left(\theta\right)^{1 - \frac{n-1}{p}}\left(\int_{B^{n-1}\left(z\left(\theta\right),~s\left(\theta\right)\right)}|\nabla_{x'} \bm{y}\left(\cdot, \theta\right)|^pdx'\right )^{\frac{1}{p}}\label{eq:le6'}
\end{eqnarray}
which in turn can be written as
\begin{eqnarray}
s\left(\theta\right)^{n-p-1}\le [C_p\left(n\right)]^p r^{-p}\int_{B^{n-1}\left(z\left(\theta\right),~s\left(\theta\right)\right)}|\nabla_{x'} \bm{y}\left(\cdot, \theta\right)|^pdx'.\label{eq:le6.1}
\end{eqnarray}
Integrating inequality \eqref{eq:le6.1} with respect to $\theta \in (0,d)$ and applying Fubini's theorem (see e.g.\ \cite[Theorem 1.124]{Fonseca-Leoni}), we derive
\begin{eqnarray}
\int_{0}^{d}s\left(\theta\right)^{n-p-1}d\theta &\le& [C_p\left(n\right)]^p r^{-p}\int_{0}^{d}\int_{B^{n-1}\left(z\left(\theta\right),~s\left(\theta\right)\right)}|\nabla_{x'} \bm{y}\left(x', \theta\right)|^pdx'd\theta\nonumber\\
&\le& [C_p\left(n\right)]^p r^{-p}\left(\int_{0}^{d}\int_{B^{n-1}\left(z\left(\theta\right),~s\left(\theta\right)\right)}|\nabla_{x'} \bm{y}\left(x', \theta\right)|^p+|\nabla_{\theta} \bm{y}\left(x', \theta\right)|^pdx'd\theta\right)\nonumber\\
&\le& [C_p\left(n\right)]^p r^{-p}\int_{\bm{y}^{-1}\left(B_{2r}\right)}|\nabla \bm{y}|^pdX.\label{eq:le6.2}
\end{eqnarray}
Next, we use Hölder's inequality with a parameter $\beta>0$.
\begin{eqnarray}
\begin{aligned}
d^{\frac{p}{n-1}} &= \left(\int_{0}^{d}s\left(\theta\right)^{\beta}s\left(\theta\right)^{-\beta}d\theta\right)^{\frac{p}{n-1}}\nonumber\\
&\underset{\text{H\"older}}{\leq}\left(\int_{0}^{d}s\left(\theta\right)^{\frac{\beta p}{p-n+1}}d\theta\right)^{\frac{p-n+1}{n-1}}\left(\int_{0}^{d}s\left(\theta\right)^{\frac{-\beta p}{n-1}}d\theta\right).\label{eq:le6.3}
\end{aligned}
\end{eqnarray}
\mbox{Choosing in particular } $\beta = \frac{\left(n-1\right)\left(p-n+1\right)}{p}$,\mbox{ we obtain }
\begin{eqnarray}
d^{\frac{p}{n-1}}&\le& \left(\int_{0}^{d}s\left(\theta\right)^{n-1}d\theta\right)^{\frac{p-n+1}{n-1}}\left(\int_{0}^{d}s\left(\theta\right)^{n-p-1}d\theta\right)\nonumber\\
&\le& \left(\int_{0}^{d}s\left(t\right)^{n-p-1}d\theta\right)\left[\mu\left(\bm{y}^{-1}\left(B_{2r}\right)\right)\right]^{\frac{p-n+1}{n-1}}.\label{eq:le10}
\end{eqnarray}
Using \eqref{eq:le6.2} in \eqref{eq:le10}, we get the estimate \eqref{eq:lem2}.
\end{proof}
\begin{remark}
The preceding estimate remains valid for $p\ge n$ under the continuity assumption on the mapping. Owing to the additional integrability condition $K^O_{\bm{y}}\in L^{n-1}$, the result strengthens the corresponding estimate established in \cite{Onninen}.
\end{remark}
\begin{remark}
The validity of the Lusin condition follows from \cite[Theorem 4.8]{Hencl-Koskela} together with estimate \eqref{eq:log3}.
\end{remark}
\begin{remark}
Almost everywhere differentiability of $\bm{y}$ is a consequence of \cite[Corollary 2.25]{Hencl-Koskela}.
\end{remark}
We are now in a position to prove the main topological result of this section. \Cref{thm2} provides an openness result in the critical distortion regime $K^{O}_{\bm{y}}\in L^{n-1}$, under the additional assumption of the Ciarlet-Nečas condition. The proof is motivated by the ideas surrounding the conjecture of
\cite{Iwaniec_Sverak}. A crucial step is to obtain sufficient regularity of the inverse mapping from the available distortion bounds.
\newline
\begin{theorem}\label{thm2}
Let $\Omega\subset\mathbb R^n$ be bounded, open and connected, and $\bm{y}\in W^{1,p}\left(\Omega; \mathbb R^n\right)$ for $p>n, n\ge 3$. In addition, assume that $\det(\nabla \bm{y}\left(X\right))> 0$ a.e., $K^{O}_{\bm{y}}\in L^{n-1}\left(\Omega\right)$, and $\int_{\Omega}\det(\nabla \bm{y}\left(X\right))dX\le \mu\left(\bm{y}\left(\Omega\right)\right)$. Then $\bm{y}$ is an open mapping.
\end{theorem}
\begin{proof}
The argument is divided into three parts.\newline
\textbf{Step I:} We begin by establishing almost everywhere differentiability of the inverse mapping. First, we show that $\bm{y}^{-1}$ is differentiable a.e. in $\bm{y}\left(\Omega\right)$. We proceed to show the Rademacher-Stepanov criteria of differentiability. Due to Rademacher-Stepanov Theorem \cite{Vaisala}, $\bm{y}^{-1}$ is differentiable a.e. in $\bm{y}\left(\Omega\right)$, provided that
\begin{align}
\lim\sup_{r\rightarrow 0}\frac{\operatorname{diam}\left[\bm{y}^{-1}\left(B\left(x, r\right)\right)\right]}{r}<\infty \mbox{ for a.e. } x\in \bm{y}\left(\Omega\right).\label{eq:ra1}
\end{align}
Since $\bm{y}$ is continuous, the preimage of an open set under $\bm{y}$ is open. Thus, the volumetric derivative $\mu'_{\bm{y}^{-1}}$ of the set function $\mu_{\bm{y}^{-1}}$ at a point $x\in\bm{y}\left(\Omega\right)$ is given by
\begin{eqnarray}
\mu'_{\bm{y}^{-1}}\left(x\right) = \lim_{r\rightarrow 0}\frac{\mu\left(\bm{y}^{-1}\left(B\left(x, r\right)\right)\right)}{\mu\left(B\left(x, r\right)\right)}.\label{eq:thm1}
\end{eqnarray}
It therefore suffices to verify condition \eqref{eq:ra1}.\newline\\
Define a set
\begin{eqnarray}
\mathbb M = \{X\in\Omega: \bm{y} \mbox{ is differentiable at } X \mbox{ and } J_{\bm{y}}\left(X\right)>0\}.
\end{eqnarray}
Furthermore, define
\begin{eqnarray}
\Psi\left(x\right) = \frac{|\left(\nabla \bm{y}\right)\left(\bm{y}^{-1}\left(x\right)\right)|^p}{J_{\bm{y}}\left(\bm{y}^{-1}\left(x\right)\right)}\chi_{\bm{y}\left(\mathbb M\right)}\left(x\right).\label{eq:ra2}
\end{eqnarray}
By Lemma \ref{lem2}, we have
\begin{eqnarray}
\operatorname{diam}\left(\bm{y}^{-1}\left(B_r\right)\right)&\le& C_p\left(n\right)r^{1-n}\left[\mu\left(\bm{y}^{-1}\left(B_{2r}\right)\right)\right]^{\frac{p-n+1}{p}}\left(\int_{\bm{y}^{-1}\left(B_{2r}\right)}|\nabla \bm{y}\left(X\right)|^p\right)^{\frac{n-1}{p}}.\nonumber\\
\label{eq:ra1'}
\end{eqnarray}
Using the general area formula (e.g.\ \cite{Hencl-Koskela}) for the function $\bm{y}$ in \eqref{eq:ra1'}, we obtain
\begin{align}
\operatorname{diam}\left[\left(\bm{y}^{-1}\left(B_r\right)\right)\right] &\le C_p\left(n\right)r^{1-n}\left[\mu\left(\bm{y}^{-1}\left(B_{2r}\right)\right)\right]^{\frac{p-n+1}{p}}\left(\int_{B_{2r}}\Psi\left(x\right)N\left(\bm{y}, \Omega, x\right)dx\right)^{\frac{n-1}{p}}\nonumber\\
&\le C_p\left(n\right)r^{1-n}\left[\mu\left(\bm{y}^{-1}\left(B_{2r}\right)\right)\right]^{\frac{p-n+1}{p}}\left(\int_{B_{2r}}\Psi\left(x\right)dx\right)^{\frac{n-1}{p}}.\label{eq:le14}
\end{align}
The final estimate uses the inequality $N\left(\bm{y}, \Omega, x\right)\le 1$ a.e. in $\Omega$, which follows from the Ciarlet-Ne\v{c}as condition.\newline
\\
Similarly, using the area formula (e.g.\ \cite{Hencl-Koskela}) for the function $\bm{y}^{-1}$, we obtain
\begin{align}
\int_{V}\Psi\left(x\right)dx &= \int_{V}|\left(\nabla \bm{y}\right)\left(\bm{y}^{-1}\left(x\right)\right)|^pJ_{\bm{y}^{-1}}\left(x\right)\chi_{\bm{y}\left(\mathbb M\right)}\left(x\right)dx\nonumber\\
&\le \int_{\bm{y}^{-1}\left(V\right)}|\nabla \bm{y}|^pdX<\infty
\end{align}
for any open subset $V$ of $\bm{y}\left(\Omega\right)$. The Lebesgue-Besicovitch differentiation theorem (cf.\ \cite{Evans-Gariepy}) for the function $\Psi$ yields
\begin{eqnarray}
\lim\sup_{r\rightarrow 0}\left(\fint_{B_{2r}}\Psi\left(x\right) dx\right)=\Psi\left(x\right)<\infty \mbox{ for a.e. } x\in B_{2r}. \label{eq:ra3} 
\end{eqnarray}
The last implication follows by using the definition of outer distortion \eqref{eq:outer_distortion} in \eqref{eq:ra2}.
\newline\\
From \eqref{eq:le14}, we get
\begin{eqnarray}
\lim\sup_{r\rightarrow 0}\frac{\operatorname{diam}\left[\bm{y}^{-1}\left(B_r\right)\right]}{r}&\le& C_p\left(n\right)\lim\sup_{r\rightarrow 0}\left(\fint_{B_{2r}}\Psi\left(x\right) dx\right)^{\frac{n-1}{p}}\left(\frac{\mu\left(\bm{y}^{-1}\left(B_{2r}\right)\right)}{\mu\left(B_{2r}\right)}\right)^{\frac{p-n+1}{p}}\nonumber\\
&\le& C_p\left(n\right)\left[\Psi\left(x\right)\right]^{\frac{n-1}{p}}\left[\mu'_{\bm{y}^{-1}\left(x\right)}\right]^{\frac{p-n+1}{p}}.\label{eq:le16}
\end{eqnarray}
Using the estimate \eqref{eq:ra3} and the Lebesgue's Theorem (e.g.\ \cite[Theorem 23.5]{Vaisala}) in \eqref{eq:le16}, we conclude that
\begin{eqnarray}
\lim\sup_{r\rightarrow 0}\frac{\operatorname{diam}\left[\bm{y}^{-1}\left(B_r\right)\right]}{r} < \infty \mbox{ for a.e. }  x \mbox{ in } \bm{y}\left(\Omega\right).\label{eq:ra4} 
\end{eqnarray}
This proves the almost everywhere differentiability of $\bm{y}^{-1}$ in $\bm{y}\left(\Omega\right)$.\newline\\
\textbf{Step II:} Having established differentiability almost everywhere, we next derive estimates for the weak gradient of the inverse mapping. 
Using \eqref{eq:ra2} in \eqref{eq:le16} and applying  \cite[Theorem 24.4.]{Vaisala} for $\bm{y}^{-1}$, we obtain
\begin{eqnarray}
|\nabla \bm{y}^{-1}\left(x\right)|\le C_p\left(n\right)\left[K^{O}_{\bm{y}}\left(\bm{y}^{-1}\left(x\right)\right)\right]^{\frac{n-1}{n}}\chi_{\bm{y}\left(M\right)}\left(x\right)\left[J_{\bm{y}^{-1}}\left(x\right)\right]^{\frac{1}{n}}.\label{eq:17}
\end{eqnarray}
Combining this estimate with the definition of the outer distortion coefficient and the assumption $K^{O}_{\bm{y}}\in L^{n-1}$, we conclude that
$\bm{y}^{-1}$ is itself a mapping of finite distortion.
\newline\\
\textbf{Step III.} It remains to verify continuity of the inverse. For this purpose, we prove the local Sobolev regularity $\bm{y}^{-1}\in W^{1,n}_{loc}\left(\bm{y}\left(\Omega\right)\setminus\bm{y}\left(\partial\Omega\right); \mathbb R^n\right)$.
Since $\bm{y}^{-1}$ has finite distortion and $\bm{y}^{-1}$ is differentiable a.e. in $\bm{y}\left(\Omega\right)$, we can define a set
\begin{eqnarray}
\mathcal{O} = \{x\in \bm{y}\left(\Omega\right): \bm{y}^{-1} \mbox{ is differentiable at } x \mbox{ and } J_{\bm{y}^{-1}}\left(x\right)>0\}.\label{eq:17.1}
\end{eqnarray}
Moreover, we have $\mu\left(\mathcal{O}\right) = \mu\left(\{x\in \bm{y}\left(\Omega\right):\nabla \bm{y}^{-1}\left(x\right)\ne 0\}\right)$. Therefore, it follows that
\begin{eqnarray}\int_{\bm{y}\left(\Omega\right)\setminus\bm{y}\left(\partial\Omega\right)}|\nabla \bm{y}^{-1}\left(x\right)|^ndx
&\leq&\int_{\bm{y}\left(\Omega\right)}|\nabla \bm{y}^{-1}\left(x\right)|^ndx\nonumber\\
&=& \int_{\mathcal{O}}|\nabla \bm{y}^{-1}\left(x\right)|^ndx\nonumber\\
&=& \int_{\mathcal{O}}|\nabla \bm{y}^{-1}\left(x\right)|^nJ_{\bm{y}}\left(\bm{y}^{-1}\left(x\right)\right)J_{\bm{y}^{-1}}\left(x\right)dx\nonumber\\
&\leq& \int_{\Omega}|\nabla \bm{y}^{-1}\left(\bm{y}\left(X\right)\right)|^nJ_{\bm{y}}\left(X\right)dX.\label{eq:17.1'}
\end{eqnarray}
Using Cramer's rule, we get
\begin{eqnarray}
\int_{\Omega}|\nabla \bm{y}^{-1}\left(\bm{y}\left(X\right)\right)|^nJ_{\bm{y}}\left(X\right)dX&=& \int_{\Omega}\frac{|\mbox{adj}\left(\nabla \bm{y}\left(X\right)\right)|^n}{\left[J_{\bm{y}}\left(X\right)\right]^{n-1}}dX\nonumber\\
&\underset{\eqref{eq:inner_distortion}}{=}& \int_{\Omega}K^I_{\bm{y}}\left(X\right)dX\nonumber\\
&\leq& \int_{\Omega}\left[\left(K^O_{\bm{y}}\right)\left(X\right)\right]^{n-1}dX<\infty.\label{eq:17.2}
\end{eqnarray}
In the last step we employed the pointwise comparison between the inner and outer distortion coefficients. By \eqref{eq:17.2},
\begin{eqnarray}
\int_{\bm{y}\left(\Omega\right)\setminus\bm{y}\left(\partial\Omega\right)}|\nabla \bm{y}^{-1}\left(x\right)|^ndx <\infty.\label{eq:17.3}
\end{eqnarray}
As $\bm{y}^{-1}\in W^{1,1}\left(\bm{y}\left(\Omega\right)\setminus\bm{y}\left(\partial\Omega\right); \mathbb R^n\right)$ (e.g.\ \cite[Proposition 2.5]{Bresciani21}), applying the Sobolev embedding Theorem (e.g.\ \cite[Theorem A.18]{Hencl-Koskela}) along with \eqref{eq:17.3}, we obtain $\bm{y}^{-1}\in W^{1,n}_{loc}\left(\bm{y}\left(\Omega\right)\setminus\bm{y}\left(\partial\Omega\right); \mathbb R^n\right) $.
The continuity theorem for finite distortion mappings (\cite[Theorem 2.3]{Hencl-Koskela}) therefore provides a continuous representative of $\bm{y}^{-1}$ in $\bm{y}\left(\Omega\right)\setminus\bm{y}\left(\partial\Omega\right)$, which satisfies that $\overline{\bm{y}\left(\Omega\right)\setminus\bm{y}\left(\partial\Omega\right)}=\bm{y}\left(\overline\Omega\right)=\overline{\bm{y}\left(\Omega\right)}$. Therefore $\bm{y}$ is an open mapping in $\bm{y}^{-1}\left(\bm{y}\left(\Omega\right)\setminus\bm{y}\left(\partial\Omega\right)\right)$. Since $\overline{\bm{y}^{-1}\left(\bm{y}\left(\Omega\right)\setminus\bm{y}\left(\partial\Omega\right)\right)}=\overline{\Omega}$, we obtain that $\bm{y}$ is an open mapping on $\Omega$.
\end{proof}

\begin{remark}
An important feature of the proof is that continuity of the inverse is obtained under a weaker Sobolev regularity requirement than would follow from a direct application of Morrey's embedding theorem. The additional finite distortion structure of $\bm{y}^{-1}$ is essential for this improvement.
\end{remark}

\section{Existence results and proof}\label{Existence results and proof}
% Keeping in mind the above examples, 
\noindent We now introduce the class of admissible configurations used in the existence analysis. The deformation component is based on mappings of finite distortion satisfying an integrability condition on the outer distortion coefficient that is compatible with the openness result established in Section 3. Throughout this section, we assume that the reference domain $\Omega\subset\mathbb R^3$ is open, bounded, and has Lipschitz boundary.\newline\\
For $p > 3,~q\ge 2$, the admissible deformation class is defined by
\begin{eqnarray}\label{eq:set_of_deformation}
\mathscr{Y} = \Bigg \{ \bm{y} \in W^{1,p} \left( \Omega; \mathbb R^3 \right) : \det( \nabla \bm{y}) > 0 \mbox{ a.e. in } \Omega,~K^O_{\bm{y}} \in L^{q}\left( \Omega \right),\nonumber\\ \int_{\Omega}\det(\nabla \bm{y}\left(X\right))dX\le \mu\left(\bm{y}\left(\Omega\right)\right),\nonumber\\~ \bm{y}=\bm{\overline{y}} \mbox{ on } \partial\Omega \mbox{ in the sense of trace}\Bigg \}.
\end{eqnarray}
The corresponding class of admissible magnetizations is given by
\begin{eqnarray}
\mathscr{M} = \Big \{ \Mvec \in W^{1,2} \left( \bm{y}\left(\Omega\right); \mathbb R^3 \right) : |\Mvec\left(\bm{y}\left(X\right)\right)|\det(\nabla \bm{y})=1 \mbox{ a.e. in } \Omega \Big \}.
\end{eqnarray}
The collection of admissible states then is defined as
\begin{eqnarray} 
\mathscr{A} = \Big \{ \left( \bm{y}, \Mvec \right) : \bm{y} \in \mathscr{Y} \mbox{ and }  \Mvec\in \mathscr{M} \Big \}.
\end{eqnarray}
We equip $\mathscr{A}$ with the following notion of convergence. We write  
\begin{eqnarray}
\left(\bm{y}_k, \Mvec_k\right)\rightarrow\left(\bm{y}, \Mvec\right) \mbox{ in } \mathscr{A}
\end{eqnarray}
whenever the convergences below are satisfied. 
\begin{subequations}
\begin{align}
\bm{y}_k &\rightharpoonup \bm{y} &\mbox{ weakly in } &W^{1,p}\left(\Omega; \mathbb R^3\right),\\
\chi_{\bm{y}_k\left(\Omega\right)}\Mvec_k &\rightarrow \chi_{\bm{y}\left(\Omega\right)}\Mvec &\mbox{ strongly in } &L^2\left(\mathbb R^3; \mathbb R^3\right),\\
\chi_{\bm{y}_k\left(\Omega\right)}\nabla \Mvec_k &\rightharpoonup \chi_{\bm{y}\left(\Omega\right)}\nabla\Mvec &\mbox{ weakly in } &L^2\left(\mathbb R^3; \mathbb R^{3\times 3}\right).
\end{align}
\end{subequations}
\begin{lemma}\label{lem3}
Any element in $\mathscr{Y}$ is a homeomorphism from $\Omega$ to $\bm{y}\left(\Omega\right)$.
\end{lemma}
\begin{proof}
Consider an arbitrary deformation $\tilde{\bm{y}} \in\mathscr{Y}$. Since $\tilde{\bm{y}}$ satisfies the Ciarlet-Nečas condition and $\det(\nabla\bm{y})>0$ a.e.\ in $\Omega$, injectivity holds a.e.\ in $\Omega$ by the result of \cite{Grandi2019}. Moreover, $\tilde{\bm{y}}$ is a nonconstant mapping of finite distortion whose outer distortion belongs to $L^q\left(\Omega\right)$ with $q \ge 2$. Therefore, \Cref{thm2} implies that $\tilde{\bm{y}}$ is an open mapping. \newline
Furthermore, \cite[Theorem 4.2]{Hencl-Koskela} yields the  Lusin $\left(\it{N}\right)$ condition. Combining this fact with \cite[Lemma 3.3]{Grandi2019} shows that $\tilde{\bm{y}}$ is injective on all of $\Omega$. Finally, the invariance domain theorem (see e.g.\ \cite{ciarlet2021mathematical}) implies that the continuous representative of $\tilde{\bm{y}}$ is a homeomorphism.
\end{proof}
\begin{remark}
Since the magnetoelastic energy is formulated in mixed Eulerian-Lagrangian coordinates, the analysis naturally involves the inverse deformation. Consequently, suitable regularity and topological properties of the inverse mapping are essential for the existence theory. The homeomorphism result established in Lemma \ref{lem3} provides the required framework and considerably simplifies the treatment of compositions involving deformation and magnetization.
\end{remark}
We are now ready to formulate the main existence result of the paper.\\
\begin{theorem}[Existence theorem]\label{thm1}
Assume that $\mathscr{A}$ is nonempty. Suppose that $\mathcal{W}\left(\nabla\bm{y}, \Mvec\right)$ satisfies \eqref{eq:poly}, \eqref{eq:coer} and \eqref{eq:beh}. Then $\mathscr{E}_{mag}\left(\bm{y}, \Mvec\right)$ admits a minimum on the set $\mathscr{A}$. 
\end{theorem}
The proof is carried out using the direct method of the calculus of variations. Its main ingredients are a compactness theorem for energy-bounded sequences and a lower semicontinuity result for the magnetoelastic energy functional. These two components are established in \Cref{compactness} and \Cref{lower_semi} and are combined in \Cref{Proof of existence theorem} to obtain the existence of minimizers.
\subsection{Compactness}\label{compactness}
\begin{theorem}[Compactness]\label{compact}
Let $\left(\bm{y}_k, \Mvec_k\right)_{k\in\mathbb N}\subset\mathscr{A}$ such that $\mathscr{E}_{mag}\left(\bm{y}_k, \Mvec_k\right)\le\Lambda$ for some constant $\Lambda>0$, and for any $k\in\mathbb N$. Then there exists $\left(\bm{y}, \Mvec\right)$ in $\mathscr{A}$ such that, up to a subsequence
\begin{align}
\left(\bm{y}_k, \Mvec_k\right)\rightarrow\left(\bm{y}, \Mvec\right) \mbox{ in } \mathscr{A}.
\end{align}
\end{theorem}
\begin{proof}
The goal is to show that every sequence with uniformly bounded energy admits a subsequence converging to an admissible limit. For the deformation component, the argument relies on compactness properties of mappings of finite distortion together with the biting convergence framework. A key feature of the present setting is that the limit deformation retains the homeomorphic character of the approximating sequence. This is a consequence of the finite distortion structure and should be contrasted with the general situation for Sobolev homeomorphisms, where injectivity may be lost under weak convergence; see Remark~\ref{re6}.\newline
Once the homeomorphism property of the limiting deformation has been established, the compactness analysis of the magnetization becomes considerably simpler. In particular, continuity of the inverse deformation allows us to handle compositions in a direct manner and avoids several technical difficulties encountered in earlier existence theories.\\
\noindent\newline \textbf{Step I: Compactness in the $\bm{y}$-component}\\
The coercivity assumption \eqref{eq:coer} immediately yields
\begin{eqnarray}
\mathscr{E}_{mag}\left(\bm{y}_k, \Mvec_k\right) \ge  \alpha_1\left(||\nabla\bm{y}_k||_{L^p} + ||K^{O}_{\bm{y}_k}||_{L^q} + \left\lVert\dfrac{1}{\det(\nabla \bm{y}_k)}\right\lVert_{L^s}\right)\nonumber\newline\\
~~~~~~+ \alpha||\nabla\Mvec_k||_{L^2} + \dfrac{\mu_0}{2}||\Hvec_{\Mvec_k}||_{L^2} + \alpha_2\label{eq:en_bound}
\end{eqnarray}
for $p>3, q\ge 2, s>1$.\newline
By \eqref{eq:en_bound}, there exists a constant $\Lambda>0$ such that 
\begin{eqnarray}
\sup_{k\in\mathbb N}||\nabla \bm{y}_k||_{L^p\left(\Omega;~\mathbb R^{3\times 3}\right)}\le \sup_{k\in\mathbb N}\mathscr{E}_{mag}\left(\bm{y}_k, \Mvec_k\right) \le \Lambda\label{eq:compact1}
\end{eqnarray}
for $p>3$.\newline\\
Choose an arbitrary deformation $\tilde{\bm{y}}\in\mathscr{Y}$. By the Poincar\'e inequality (see e.g.\ \cite[Section 4, Theorem 2]{Evans-Gariepy}), there exists a constant $C\left(p, n\right)>0$ (depending on $p$ and $n$) such that
\begin{eqnarray}
\sup_{k\in\mathbb{N}}||\bm{y}_k||_{L^p\left(\Omega;~\mathbb R^3\right)}\le C\left(p, n\right)||\nabla \bm{y}_k - \nabla \tilde{\bm{y}}||_{L^p\left(\Omega;~\mathbb R^3\right)} + \tilde{C}\le C_1\label{eq;compact2}
\end{eqnarray}
for some constant $C_1>0$.
An application of the Banach-Alaoglu theorem therefore provides a subsequence, not relabelled, such that
\begin{eqnarray}
\bm{y}_k\rightharpoonup \bm{y} \mbox{ weakly in } W^{1, p}\left(\Omega; \mathbb R^3\right).\label{eq:compact3}
\end{eqnarray}
It remains to verify that the weak limit belongs to the admissible deformation class $\mathscr{Y}$. Due to the coerciveness \eqref{eq:coer} in $\mathcal{W}$, 
\begin{eqnarray}
\sup_{k\in\mathbb N}\left\lVert K^O_{{\bm{y}}_k}\right\lVert_{L^q\left(\Omega\right)}<\infty \mbox{ for } q\ge 2.\label{eq:compact4}
\end{eqnarray}By the Banach-Alaoglu Theorem, up to a subsequence
\begin{eqnarray}
K^O_{{\bm{y}}_k}\rightharpoonup\tilde{K} \mbox{ weakly in } L^q\left(\Omega\right) \mbox{ for } q\ge 2.\label{eq:compact5}
\end{eqnarray}
Using the property of biting convergence \cite{Gehring1999} and the lower semicontinuity of norm, we have
\begin{eqnarray}
||K^O_{\bm{y}}||_{L^q}\le||\tilde{K}||_{L^q}\le\liminf_{k\rightarrow\infty}||K_{\bm{y}_k}||_{L^q}<\infty.
\end{eqnarray}
In the above, the third implication follows from the weak lower-semicontinuity property of the norm.\newline\\
The next step is to establish positivity of the Jacobian determinant almost everywhere in $\Omega$. Due to the weak continuity of Jacobian minors (see e.g.\ \cite[Theorem~8.20]{dacorogna2007direct}), 
\begin{eqnarray}
\det(\nabla \bm{y}_k)\rightharpoonup\det(\nabla \bm{y}) \mbox{ weakly in } L^{p/3}\left(\Omega\right) \mbox{ for } p>3.\label{eq:compact10}
\end{eqnarray}
Then for any measurable set $V\subset\Omega$, it follows that
\begin{eqnarray}
\int_{V}\det(\nabla \bm{y})dX=\lim_{k\rightarrow\infty}\int_{V}\det(\nabla \bm{y}_k)dX\ge 0.\label{eq:compact11}
\end{eqnarray}
Therefore, $\det(\nabla \bm{y})\ge 0$ in $\Omega$. Suppose, for contradiction, that $\det(\nabla \bm{y})=0$ on a nonzero measurable set $V_0\subset\Omega$. Then, up to subsequences
\begin{eqnarray}
\det(\nabla \bm{y}_k)\rightarrow 0 \mbox{ a.e. in } V_0,\label{eq:compact12}\\
\frac{1}{\det(\nabla \bm{y}_k)}\rightarrow +\infty \mbox{ a.e. in } V_0.\label{eq:compact13}
\end{eqnarray}
Applying Fatou's lemma (e.g.\ \cite[Section 1, Theorem 1]{Evans-Gariepy}) in \eqref{eq:compact13}, we obtain
\begin{eqnarray}
\liminf_{k\rightarrow\infty}\int_{V_0}\frac{1}{\det(\nabla \bm{y}_k)}dX=+\infty,\label{eq:compact14}
\end{eqnarray}
which contradicts the fact that the energy functional is bounded. Hence, $\det(\nabla \bm{y})>0$ a.e. in $\Omega$. Again, due to \eqref{eq:compact10},
\begin{eqnarray}
\sup_{k\in\mathbb N}||\det(\nabla \bm{y}_k)||_{L^{\frac{p}{3}}\left(\Omega\right)}<\infty \mbox{ for } p>3.\label{eq:compact15}
\end{eqnarray}
By Proposition 2.27 in \cite{Fonseca-Leoni}, together with \eqref{eq:compact15}, it follows that
\begin{eqnarray}
\big\{\det(\nabla \bm{y}_k): \Omega\rightarrow [0,\infty] \mbox{ for } k\in\mathbb N\big\} \mbox{ is } \frac{p}{3} \mbox{-equi-integrable for } p>3.\label{eq:compact16}
\end{eqnarray}
For every $k\in \mathbb N$, it holds that
\begin{eqnarray}
\int_{\Omega}|\det(\nabla\bm{y}_k)|dX\le\mu\left(\bm{y}_k\left(\Omega\right)\right).\label{eq:compact17}
\end{eqnarray}
Due to claim \eqref{eq:compact16}, passing to the limit as $k\rightarrow\infty$ in \eqref{eq:compact17}, we obtain
\begin{eqnarray}
\int_{\Omega}|\det(\nabla\bm{y})|dX\le\mu\left(\bm{y}\left(\Omega\right)\right).\label{eq:compact18}  
\end{eqnarray}
To identify the boundary values of the limit deformation, we employ the compactness of the trace operator $\gamma: W^{1,\frac{3}{2}}\rightarrow L^2\left(\partial\Omega\right)$ (see e.g.\ \cite[Theorem 6.1-7]{ciarlet2021mathematical}) to get
\begin{eqnarray}
\gamma\left(\bm{y}_k\right)\rightarrow\gamma\left(\bm{y}\right) \mbox{ strongly in } L^2\left(\partial\Omega\right).
\end{eqnarray}
Up to a subsequence,
\begin{eqnarray}
\gamma\left(\bm{y}_k\right)\rightarrow\gamma\left(\bm{y}\right) \mbox{ a.e. in } \Omega.
\end{eqnarray}
Hence, $\bm{y}=\overline{\bm{y}}$ on $\partial\Omega$ in the sense of trace. Consequently, the admissible deformation class $\mathscr{Y}$ is sequentially weakly closed.
\newline
\begin{remark}\label{re6} In general, weak convergence in Sobolev spaces does not preserve the homeomorphism property. Molchanova and Vodopyanov \cite{Molchanova} provides an example:  $\phi_k\left(x\right) = |x|^{k-1}x$ is a sequence of homeomorphic functions defined on a punctured unit ball, and the limit mapping $\phi_0\left(x\right) = 0$ is not injective.
\end{remark}
\noindent \textbf{Step II: Compactness in $\Mvec$-component}\\
Due to Morrey's embedding, cf. \cite{ciarlet2021mathematical}, up to a subsequence
\begin{eqnarray}
\bm{y}_k\rightarrow\bm{y} \mbox{ uniformly in } C\left(\overline{\Omega}; \mathbb R^3\right).\label{eq:mcom1}
\end{eqnarray}
For $\delta > 0$ sufficiently small, consider a collection of domains 
\begin{displaymath}
O_{\delta}\left(\bm{y}\right) = \{ z \in \bm{y}\left(\Omega\right): \operatorname{dist}\left(z, \partial \bm{y}\left(\Omega\right) \right) > \delta \}. %\left(\delta>0 \mbox{ is very small }\right).
\end{displaymath} Thanks to the Lindel\"of Theorem \cite[Theorem A.8]{Fonseca-Leoni}, we can cover $\bm{y}\left(\Omega\right)$ such that
\begin{eqnarray}
\bigcup_{n\in\mathbb N}O_{\delta=\frac{1}{n}}\left(\bm{y}\right) = \bm{y}\left(\Omega\right).
\end{eqnarray}
The uniform convergence in \eqref{eq:mcom1} implies that, for every fixed $\delta >0$, one may assume that $O_{\delta}\left(\bm{y}\right)\subset\bm{y}_k\left(\Omega\right)$ for all sufficiently large $k$. We first prove that, for every fixed $\delta > 0$, there exists a subsequence (without re-indexing) such that
\begin{eqnarray}
\Mvec_k \rightharpoonup \Mvec \mbox{ weakly in } W^{1,2}\left(O_{\delta}\left(\bm{y}\right); \mathbb R^3\right).\label{eq:mcom3}
\end{eqnarray}
The boundedness of the magnetization sequence follows from the saturation constraint together with the energy estimate. Indeed, for any fixed $\delta > 0$, it holds that
\begin{eqnarray}
\sup_{k\in\mathbb N}\int_{O_{\delta}\left(\bm{y}\right)} |\Mvec_k\left(x\right)|^2 dx
&\le& \sup_{k\in\mathbb N}\int_{\bm{y}\left(\Omega\right)}|\Mvec_k\left(x\right)|^2dx\nonumber\\ &\le&\sup_{k\in\mathbb N} \int_{\Omega}|\Mvec_k\left(\bm{y}_k\left(X\right)\right)|^2 \det(\nabla \bm{y}_k\left(X\right))dX\nonumber\\
&\underset{\eqref{eq:sat1}}{\leq}&\sup_{k\in\mathbb N} \int_{\Omega} \dfrac{1}{\det(\nabla \bm{y}_k\left(X\right))}dX <\infty.\label{eq:mcom4}
\end{eqnarray}
The above bound is obtained using \eqref{eq:sat1} followed by applying the coerciveness \eqref{eq:coer} in $\mathcal{W}$. Applying the Banach-Alaoglu theorem, we obtain, up to a subsequence 
\begin{eqnarray}
\Mvec_k \rightharpoonup \Mvec \mbox{ weakly in } W^{1,2}\left( O_{\delta}\left(\bm{y}\right); \mathbb R^3\right).\label{eq:mcom5}
\end{eqnarray}
Due to Rellich-Kondrachov embedding, there exists a subsequence (without re-indexing) such that 
\begin{eqnarray}
\Mvec_k \rightarrow \Mvec \mbox{ strongly in } L^2\left( O_{\delta}\left(\bm{y}\right); \mathbb R^3\right).\label{eq:mcom5'}
\end{eqnarray}
We next establish the convergences
\begin{eqnarray}
\chi_{\bm{y}_k\left(\Omega\right)}\Mvec_k &\rightarrow \chi_{\bm{y}\left(\Omega\right)}\Mvec &\mbox{ strongly in } L^2\left( \mathbb R^3; \mathbb R^3\right),\label{eq:mcom6}\\
\chi_{\bm{y}_k\left(\Omega\right)}\nabla \Mvec_k &\rightharpoonup \chi_{\bm{y}\left(\Omega\right)}\nabla \Mvec &\mbox{ weakly in } L^2\left( \mathbb R^3; \mathbb R^{3\times 3}\right).\label{eq:mcom7}
\end{eqnarray}
\noindent To show \eqref{eq:mcom6}, we decompose the difference according to
\begin{eqnarray} 
& &\chi_{\bm{y}_k\left(\Omega\right)} \Mvec_k - \chi_{\bm{y}\left(\Omega\right)}\Mvec\nonumber\\
&=& \left(\chi_{\bm{y}_k\left(\Omega\right)} - \chi_{O_{\delta}\left(\bm{y}\right)}\right)\Mvec_k + \chi_{O_{\delta}\left(\bm{y}\right)}\left(\Mvec_k - \Mvec\right) + \left(\chi_{O_{\delta}\left(\bm{y}\right)} - \chi_{\bm{y}\left(\Omega\right)}\right)\Mvec.\label{eq:mcom8}
\end{eqnarray}
Therefore,
\begin{eqnarray}
& &||\chi_{\bm{y}_k\left(\Omega\right)} \Mvec_k - \chi_{\bm{y}\left(\Omega\right)}\Mvec ||_{L^2\left(\mathbb R^3; \mathbb R^3\right)}\nonumber\\
&\le& \underbrace{||\Mvec_k||_{L^2\left(\bm{y}_k\left(\Omega\right)\setminus O_{\delta}\left(\bm{y}\right)\right)}}_{\mathcal{I}_1} + \underbrace{||\Mvec_k - \Mvec||_{L^2\left( O_{\delta}\left(\bm{y}\right)\right)}}_{\mathcal{I}_2} + \underbrace{||\Mvec||_{L^2\left(\bm{y}\left(\Omega\right)\setminus O_{\delta}\left(\bm{y}\right)\right)}}_{\mathcal{I}_3}.
\end{eqnarray}
To estimate $\mathcal{I}_1$, we set $\Omega^{\delta}_{k} = \bm{y}^{-1}_k\left(O_{\delta}\left(\bm{y}\right)\right)$ and write
\begin{eqnarray}\label{eq:com1}
\mathcal{I}_1
&=&||\Mvec_k||^{2}_{L^2\left(\bm{y}_k\left(\Omega\right)\right)}\setminus O_{\delta}\left(\bm{y}\right)\}\nonumber\\
&=&\int_{\Omega}\left(1-\chi_{\Omega^{\delta}_{k}}\right)|\Mvec_k\left(\bm{y}_k\left(X\right)\right)|^2\det(\nabla \bm{y}_k)dX\nonumber\\
&\underset{\eqref{eq:sat1}}{=}&\int_{\Omega}\dfrac{1-\chi_{\Omega^{\delta}_{k}}}{\det(\nabla \bm{y}_k)}dX.
\end{eqnarray}
%$$ = \int_{\Omega} \dfrac{\left(1-\chi_{\Omega^{\epsilon}_{k}}\right)}{\det(\nabla y_k)}  dX $$
%$$= \int_{\Omega}|\Omega \setminus y^{-1}_k\left(O_{\epsilon}\left(y\right)\right)| $$The last implication in \eqref{eq:com1} follows using \eqref{eq:sat1}. In what follows, we show that $\bm{y}^{-1}_k\left(O_{\delta}\left(\bm{y}\right)\right)$ uniformly converges to $\bm{y}^{-1}\left(O_{\delta}\left(\bm{y}\right)\right)$ for any $\delta > 0$. We consider a compact set $\mathcal{C}$ such that $O_{\delta}\left(\bm{y}\right) \subset \mathcal{C} \subset \bm{y}\left(\Omega\right)$. It is sufficient to show that $\bm{y}^{-1}_k$ uniformly converges to $\bm{y}^{-1}$ on $\mathcal{C}$. Let us consider a point $X\in\Omega$ such that $\bm{y}_k\left(X\right)\rightarrow \bm{y}\left(X\right)$. We denote $\bm{y}_k\left(X\right)$ by $x_k$ and $\bm{y}\left(X\right)$ by $x$. Due to continuity of $\bm{y}^{-1}_k$, $x_k\rightarrow x$ in $\mathcal{C}$ implies that $\bm{y}^{-1}_k\left(x_k\right)\rightarrow \bm{y}^{-1}_k\left(x\right)$ and the following estimate holds
The last implication in \eqref{eq:com1} follows using \eqref{eq:sat1}. We now verify the convergence of the inverse images appearing in the first term. More precisely, we show that $\bm{y}^{-1}_k\left(O_{\delta}\left(\bm{y}\right)\right)$ converges uniformly to $\bm{y}^{-1}\left(O_{\delta}\left(\bm{y}\right)\right)$ for any $\delta > 0$. We consider a compact set $\mathcal{C}$ such that $O_{\delta}\left(\bm{y}\right) \subset \mathcal{C} \subset \Omega^{\bm{y}}$. It suffices to establish uniform convergence of the inverse mappings on the compact set $\mathcal{C}$. Let us consider a point $X\in\Omega$ such that $\bm{y}_k\left(X\right)\rightarrow \bm{y}\left(X\right)$. We denote $\bm{y}_k\left(X\right)$ by $x_k$ and $\bm{y}\left(X\right)$ by $x$. Due to continuity of $\bm{y}^{-1}_k$, $x_k\rightarrow x$ in $\mathcal{C}$ implies that $\bm{y}^{-1}_k\left(x_k\right)\rightarrow \bm{y}^{-1}_k\left(x\right)$ and the following estimate holds
\begin{eqnarray}
\sup_{x\in \mathcal{C}}|\bm{y}^{-1}_k\left(x\right) - \bm{y}^{-1}\left(x\right)|
&=&\sup_{x\in \mathcal{C}}|\bm{y}^{-1}_k\left(x\right) - \bm{y}^{-1}_k\left(x_k\right) + \bm{y}^{-1}_k\left(x_k\right) - \bm{y}^{-1}\left(x\right)|\nonumber\\
&\le&\sup_{x\in \mathcal{C}}|\bm{y}^{-1}_k\left(x\right) -\bm{y}^{-1}_k\left(x_k\right)| + \sup_{x\in \mathcal{C}}|\bm{y}^{-1}_k\left(x_k\right) - \bm{y}^{-1}\left(x\right)|\nonumber\\
&<&\epsilon\label{eq:mcom9}
\end{eqnarray}
for any $\epsilon>0$. Therefore, $\bm{y}^{-1}_k\left(O_{\delta}\left(\bm{y}\right)\right)$ uniformly converges to $\bm{y}^{-1}\left(O_{\delta}\left(\bm{y}\right)\right)$ for any $\delta > 0$. Again, from \eqref{eq:com1} we have the estimate
\begin{eqnarray}
\mathcal{I}_1&=&||\Mvec_k||^{2}_{L^2\left(\bm{y}_k\left(\Omega\right)\setminus O_{\delta}\left(\bm{y}\right)\right)}\nonumber\\
&=& \int_{\Omega} \dfrac{1-\chi_{\Omega^{\delta}_{k}}}{\det(\nabla \bm{y}_k)}dX\nonumber\\
&=& \int_{\Omega}\dfrac{1}{\det(\nabla \bm{y}_k)}\left(\chi_{\Omega}-\chi_{\Omega^{\delta}_{k}}\right)dX\nonumber\\
&\underset{\text{H\"older}}{\leq}&\sup_{k\in\mathbb N}\bigg\{\left\lVert\dfrac{1}{\det(\nabla \bm{y}_k)}\right\lVert_{L^1\left(\Omega\right)}\left\lVert\chi_{\Omega^{\delta}_{k}}-\chi_{\Omega}\right\lVert_{L^{\infty}\left(\Omega\right)}\bigg\}.\label{eq:mcom10}
\end{eqnarray}
Since $\bm{y}^{-1}_k\left(O_{\delta}\left(\bm{y}\right)\right)$ uniformly converges to $\bm{y}^{-1}\left(O_{\delta}\left(\bm{y}\right)\right)$ for any $\delta > 0$, we obtain
\begin{eqnarray}
\lim_{\delta\rightarrow 0}\lim_{k\rightarrow\infty}\left\lVert\chi_{\Omega^{\delta}_{k}}-\chi_{\Omega}\right\lVert_{L^{\infty}}= 0\label{eq:linfty}
\end{eqnarray}
%as $k\rightarrow\infty$ and $\delta\rightarrow 0$.\newline\\
Using $\sup_{k\in\mathbb N}\left\lVert\dfrac{1}{\det(\nabla \bm{y}_k)}\right\lVert_{L^1\left(\Omega\right)}<\infty$ by \eqref{eq:coer} together with \eqref{eq:linfty}, it follows from \eqref{eq:mcom10} that
\begin{eqnarray}
\lim_{\delta\rightarrow 0}\lim_{k\rightarrow\infty}\mathcal{I}_1=\lim_{\delta\rightarrow 0}\lim_{k\rightarrow\infty}||\Mvec_k||^{2}_{L^2\left(\bm{y}_k\left(\Omega\right)\setminus O_{\delta}\left(\bm{y}\right)\right)}= 0. \label{eq:mcom11}
\end{eqnarray}
%as $k\rightarrow\infty$ and $\delta\rightarrow 0$.\newline\\
By \eqref{eq:mcom5'}, we have
\begin{eqnarray}
\lim_{k\rightarrow\infty}\mathcal{I}_2=  \lim_{k\rightarrow\infty}||\Mvec_k - \Mvec||_{L^2\left( O_{\delta}\left(\bm{y}\right)\right)}= 0.\label{eq:mcom12}
\end{eqnarray}
%as $k\rightarrow\infty$.\newline\\
Because $\lim_{\delta\rightarrow 0}\mu\left(\bm{y}\left(\Omega\right)\setminus O_{\delta}\left(\bm{y}\right)\right)= 0$, and the integral $\int_{\bm{y}\left(\Omega\right)\setminus O_{\delta}\left(\bm{y}\right)}|\Mvec|^2dx$ is absolutely continuous with respect to the measure $\mu_x$,
\begin{eqnarray}
\lim_{\delta\rightarrow 0}\mathcal{I}_3=\lim_{\delta\rightarrow 0}\int_{\bm{y}\left(\Omega\right)\setminus O_{\delta}\left(\bm{y}\right)}|\Mvec|^2dx= 0.\label{eq:mcom13}
\end{eqnarray}
Combining \eqref{eq:mcom11}, \eqref{eq:mcom12} and \eqref{eq:mcom13}, we finally obtain \eqref{eq:mcom6}.\newline
\\
\noindent The proof of \eqref{eq:mcom7} follows the same strategy. For any arbitrary $\bm{\psi}\in L^2\left( \mathbb R^3; \mathbb R^{3\times 3}\right)$, we can estimate
\begin{eqnarray}
& &\int_{\mathbb R^3} \left( \chi_{\bm{y}_k\left(\Omega\right)} \nabla \Mvec_k - \chi_{\bm{y}\left(\Omega\right)}\nabla\Mvec \right)\cdot\bm{\psi} dx\nonumber\\
&=&\int_{\mathbb R^3} \Big[\left( \chi_{\bm{y}_k\left(\Omega\right)} - \chi_{O_{\delta}{\left(\bm{y}\right)}} \right) \nabla \Mvec_k\cdot\bm{\psi} + \chi_{O_{\delta}{\left(\bm{y}\right)}}\left(\nabla \Mvec_k - \nabla \Mvec\right)\cdot\bm{\psi}\nonumber\\
& &~~~~~~~+ \left(\chi_{O_{\delta}{\left(\bm{y}\right)}} - \chi_{\bm{y}\left(\Omega\right)}\right)\nabla\Mvec\cdot\bm{\psi}\Big]dx\nonumber\\ 
&=& \mathcal{I}_4+\mathcal{I}_5+\mathcal{I}_6.\label{eq:mcom14}
\end{eqnarray}
Using H\"older's inequality, we have
\begin{eqnarray}
\mathcal{I}_4&=&\int_{\mathbb R^3} \left(\chi_{\bm{y}_k\left(\Omega\right)} - \chi_{O_{\delta}{\left(\bm{y}\right)}} \right) \nabla \Mvec_k\cdot\bm{\psi} dx\nonumber\\
&\le&\sup_{k\in\mathbb N} ||\nabla \Mvec_k||_{L^2\left(\left(\bm{y}_k\left(\Omega\right)\right)\setminus O_{\delta}{\left(\bm{y}\right)}; \mathbb R^{3\times 3}\right)}||\bm{\psi}||_{L^2\left(\bm{y}_k\left(\Omega\right)\setminus O_{\delta}{\left(\bm{y}\right)}; \mathbb R^{3\times 3}\right)}.\label{eq:mcom15}
%&\le& C\sup_{k\in\mathbb N}||\bm{\psi}||_{L^2\left(\Omega^{\bm{y}_k}\setminus O_{\delta}{\left(\bm{y}\right)}\right)}
\end{eqnarray}
Since $\lim_{\delta\rightarrow 0}\lim_{k\rightarrow\infty}\mu\left(\bm{y}_k\left(\Omega\right)\setminus O_{\delta}\left(\bm{y}\right)\right)=0$, and the integral $\int_{\bm{y}_k\left(\Omega\right)\setminus O_{\delta}\left(\bm{y}\right)}|\bm{\psi}|^2dx$ is absolutely continuous with respect to the measure $\mu_x$, we have
\begin{eqnarray}
\lim_{\delta\rightarrow 0}\lim_{k\rightarrow\infty}||\bm{\psi}||_{L^2\left(\bm{y}_k\left(\Omega\right)\setminus O_{\delta}{\left(\bm{y}\right)}; \mathbb R^{3\times 3}\right)}\rightarrow 0.\label{eq:mcom16}
\end{eqnarray}
%as $k\rightarrow\infty$ \mbox{ and } $\delta\rightarrow 0$.\newline\\
Using $\sup_{k\in\mathbb N} ||\nabla \Mvec_k||_{L^2\left(\bm{y}_k\left(\Omega\right)\setminus O_{\delta}{\left(\bm{y}\right)}; \mathbb R^{3\times 3}\right)}<\infty$ by \eqref{eq:coer} together with \eqref{eq:mcom16}, we conclude from \eqref{eq:mcom15} that
\begin{eqnarray}
\lim_{\delta\rightarrow 0}\lim_{k\rightarrow\infty}\mathcal{I}_4=\lim_{\delta\rightarrow 0}\lim_{k\rightarrow\infty}\int_{\mathbb R^3} \left( \chi_{\bm{y}_k\left(\Omega\right)} - \chi_{O_{\delta}{\left(\bm{y}\right)}} \right) \nabla \Mvec_k\cdot\bm{\psi} dx=0.
\end{eqnarray}
By \eqref{eq:mcom5'}, we immediately have
\begin{eqnarray}
\lim_{k\rightarrow\infty}\mathcal{I}_5=\lim_{k\rightarrow\infty}\int_{\mathbb R^3}\chi_{O_{\delta}\left(\bm{y}\right)}\left(\nabla \Mvec_k - \nabla \Mvec\right)\cdot\bm{\psi}dx = 0\label{eq:mcom17}
\end{eqnarray}
for any $\bm{\psi}\in L^2\left( \mathbb R^3; \mathbb R^{3\times 3}\right)$ and $\delta>0$.\newline\\
Repeating the argument used in the case of $\mathcal{I}_4$, we obtain
\begin{eqnarray}
\lim_{\delta\rightarrow 0}\mathcal{I}_6=\lim_{\delta\rightarrow 0}\int_{\mathbb R^3}\left(\chi_{O_{\delta}{\left(\bm{y}\right)}} - \chi_{\bm{y}\left(\Omega\right)}\right)\nabla\Mvec\cdot\bm{\psi}dx=0\label{eq:i6}
\end{eqnarray}
for any $\bm{\psi}\in L^2\left( \mathbb R^3; \mathbb R^{3\times 3}\right)$.
Inserting \eqref{eq:mcom15}, \eqref{eq:mcom17} and \eqref{eq:i6} into \eqref{eq:mcom14}, the assertion \eqref{eq:mcom7} follows.\newline\\
%as $\delta\rightarrow 0$.\newline\\
\textbf{Step III.}\newline
To conclude the compactness proof, it remains to verify that the saturation constraint is preserved in the limit.
\begin{eqnarray}
|\Mvec\left(\bm{y}\left(X\right)\right)|\det(\nabla \bm{y}\left(X\right)) = 1 \mbox{ a.e. in } \Omega.\label{eq:mcom18}
\end{eqnarray}
Fix a point $X\in\Omega$ and consider consider a family of nested balls $B\left(X, r_i\right)$ of radii $r_i$ such that
\begin{eqnarray}
\Omega = \bigcup_{i=1}^{\infty}B\left(X, r_i\right)=\bigcup_{i=1}^{\infty}\mathcal{B}_i.\label{eq:mcom19}
\end{eqnarray}
The Lindelöf property (see e.g.\ \cite{Fonseca-Leoni}) implies that \eqref{eq:mcom19} holds. Each open ball $B\left(X, r_i\right)$ is denoted by $\mathcal{B}_i$ in \eqref{eq:mcom19}. Due to the energy bound,
\begin{eqnarray}
\sup_{k\in\mathbb N}||\Mvec_k||_{W^{1, 2}\left(\bm{y}_k\left(\mathcal{B}_i\right);~\mathbb R^3\right)}<\infty\label{eq:mcom20}
\end{eqnarray}
for each $i\in\mathbb N$. Thus, up to a subsequence (without re-indexing)
\begin{eqnarray}
\chi_{\bm{y}_k\left(\mathcal{B}_i\right)} \Mvec_k \rightarrow \chi_{\bm{y}\left(\mathcal{B}_i\right)} \Mvec \mbox{ strongly in } L^2\left(\mathbb R^3; \mathbb R^3\right)\label{eq:mcom21} 
\end{eqnarray}
for each $i\in\mathbb N$. Since $\mu\left(\mathcal{B}_i\right)<\infty$, we have
\begin{eqnarray}
\chi_{\bm{y}_k\left(\mathcal{B}_i\right)} \Mvec_k \rightarrow \chi_{\bm{y}\left(\mathcal{B}_i\right)} \Mvec \mbox{ strongly in } L^1\left(\mathbb R^3; \mathbb R^3\right)\label{eq:mcom22}
\end{eqnarray}
for each $i\in\mathbb N$. Therefore,
\begin{eqnarray}
\lim_{k\rightarrow\infty}\int_{\mathcal{B}_i} |\Mvec_k\left(\bm{y}_k\left(X\right)\right)|\det(\nabla \bm{y}_k\left(X\right))dX = \int_{\mathcal{B}_i} |\Mvec\left(\bm{y}\left(X\right)\right)|\det(\nabla \bm{y}\left(X\right))dX.\label{eq:mcom23}
\end{eqnarray}
Applying \eqref{eq:sat1} in \eqref{eq:mcom23}, it holds that 
\begin{eqnarray}
\int_{\mathcal{B}_i} |\Mvec\left(\bm{y}\left(X\right)\right)|\det(\nabla \bm{y}\left(X\right))dX = \mu\left(\mathcal{B}_i\right).\label{eq:mcom24}
\end{eqnarray}
Passing to the limit as $r_i\rightarrow 0$ in \eqref{eq:mcom24} and applying the Lebesgue-Besicovitch differentiation Theorem (cf. \cite{Evans-Gariepy}), it immediately follows that
\begin{eqnarray}
|\Mvec\left(\bm{y}\left(X\right)\right)|\det(\nabla \bm{y}\left(X\right)) = 1 \mbox{ a.e. in } \Omega.\label{eq:mcom25}
\end{eqnarray}
\end{proof}
%\begin{remark}\label{re3} $|E^{\bm{y}}|$ is the measure theoretic image of $E$ under $\bm{y}$. However, in this case, the measure-theoretic images are reduced to the usual images of $E$ through the continuous representative of $\bm{y}$.
%\end{remark}
%\medskip
%\begin{remark}\label{re4} The second equality is not true in general, here it holds because our deformation functions are homeomorphic due to Lemma \ref{lem3}. This is important to note that the application of Lemma \ref{lem3} is a crucial step in the above proof, whereas Lemma \ref{3} is itself dependent on Theorem \ref{thm2}. This way the argument of compactness for the set $\mathscr{Y}$ is extended to the sharpness condition $\left(q\ge n-1\right)$.
%\end{remark}
%\medskip
%\begin{remark}\label{re5} In general, if a sequence of Sobolev homeomorphism $\bm{f}_k : \Omega \rightarrow \mathbb R^3$ with positive Jacobian determinants, $J\left(x, \bm{f}_k\right) > 0$, converges weakly to $\bm{f}$ in $W^{1,p}\left(\Omega; \mathbb R^3\right)$, then $J\left(x, \bm{f}\right) \ge 0$ a.e. in $\Omega$. But in the case of a homeomorphic function with finite distortion, we have $J\left(x, \bm{f}\right) > 0$ a.e. in $\Omega$.
%\end{remark}
%\medskip
\medskip
In the following, we prove the strong convergence for the magnetization sequence in a composition with deformation, and this convergence will allow us to establish the weak semicontinuity in the magnetostrictive function $\mathcal{W}$.
\begin{theorem}\label{saturation}Let $\left(\bm{y}_k, \Mvec_k\right)\rightarrow\left(\bm{y}, \Mvec\right)$ in $\mathscr{A}$. Then the following convergences hold:\newline
\mbox{ up to a subsequence (without re-indexing)}
\begin{eqnarray}
\Mvec_k\circ \bm{y}_k\rightarrow \Mvec\circ \bm{y} \mbox{ strongly in } L^1\left(\Omega; \mathbb R^3\right),\label{eq:mcom26}
\end{eqnarray}
\mbox{ up to a subsequence (without re-indexing)}
\begin{eqnarray}
\Mvec_k\circ \bm{y}_k\rightarrow \Mvec\circ \bm{y} \mbox{ a.e. in } \Omega.\label{eq:mcom26'}
\end{eqnarray}
\end{theorem}
\begin{proof}
The saturation constraint \eqref{eq:sat1}, together with its limiting counterpart \eqref{eq:mcom25},implies 
\begin{eqnarray}
||\Mvec_k\circ \bm{y}_k \det(\nabla \bm{y}_k)||^p_{L^p\left(\Omega;~\mathbb R^3\right)} \rightarrow ||\Mvec\circ \bm{y} \det(\nabla \bm{y})||^p_{L^p\left(\Omega;~\mathbb R^3\right)}\label{eq:mcom27}
\end{eqnarray}
for $1\leq p<\infty$.
By another application of the Banach-Alaoglu theorem, we obtain 
\begin{eqnarray}
\Mvec_k\circ \bm{y}_k \det(\nabla \bm{y}_k) \rightharpoonup \Mvec\circ \bm{y} \det(\nabla \bm{y}) \mbox{ weakly in } L^{p}\left(\Omega; \mathbb R^3\right)\label{eq:mcom28}
\end{eqnarray}
for $1 <  p < \infty$. Since weak convergence is accompanied by convergence of the corresponding norms, we deduce that
\begin{eqnarray}
\Mvec_k\circ \bm{y}_k \det(\nabla \bm{y}_k) \longrightarrow \Mvec\circ \bm{y} \det(\nabla \bm{y}) \mbox{ strongly in } L^p\left(\Omega; \mathbb R^3\right)\label{eq:mcom29}
\end{eqnarray}
for $1 \leq p < \infty$. Passing to a subsequence if necessary,
\begin{eqnarray}
\Mvec_k\circ \bm{y}_k \det(\nabla \bm{y}_k) \longrightarrow \Mvec\circ \bm{y} \det(\nabla \bm{y}) \mbox{ a.e. in } \Omega.\label{eq:com-ptwise}
\end{eqnarray}
Combining the strong convergence property \eqref{eq:mcom6} with the saturation condition  \eqref{eq:sat1}, we obtain
\begin{eqnarray}
\left\lVert\frac{1}{\det(\nabla \bm{y}_k)}\right\rVert_{L^1\left(\Omega\right)}\rightarrow \left\lVert\frac{1}{\det(\nabla \bm{y})}\right\rVert_{L^1\left(\Omega\right)}.\label{eq:norm-cov}
\end{eqnarray}
The weak convergence of the Jacobian determinant established in \eqref{eq:compact10} ogether with convergence of the norms in \eqref{eq:norm-cov}, implies
\begin{eqnarray}
\dfrac{1}{\det(\nabla \bm{y}_k)} \rightarrow \dfrac{1}{\det(\nabla \bm{y})}  \mbox{ strongly in } L^{1}\left(\Omega\right).\label{eq:L1-conv}
\end{eqnarray}
After extraction of a subsequence, we obtain
\begin{eqnarray}
\dfrac{1}{\det(\nabla \bm{y}_k)} \rightarrow \dfrac{1}{\det(\nabla \bm{y})} \mbox{ a.e. in } \Omega.\label{eq:ptwise}
\end{eqnarray}
The uniform integrability provided by the energy estimate \eqref{eq:en_bound},together with the almost everywhere convergence in \eqref{eq:ptwise}, allows us to apply Vitali's convergence theorem (see e.g.\ \cite[Theorem 2.24]{Fonseca-Leoni}) and conclude that
\begin{eqnarray}
\dfrac{1}{\det(\nabla \bm{y}_k)} \rightarrow \dfrac{1}{\det(\nabla \bm{y})} \mbox{ strongly in } L^{p/3}\label{eq:con_Lp3}
\end{eqnarray}
for $p>3$. We now estimate the difference $\Mvec_k\circ \bm{y}_k - \Mvec\circ \bm{y}$. Using H\"older's inequality with exponent $s=\frac{p}{p-3}$, $p>3$, we obtain
\begin{eqnarray}
& &||\Mvec_k\circ \bm{y}_k - \Mvec\circ \bm{y}||_{L^1}\nonumber\\
&\le&\left\lVert\Mvec_k\circ \bm{y}_k \det(\nabla \bm{y}_k)\dfrac{1}{\det(\nabla \bm{y}_k)}-\Mvec_k\circ \bm{y}_k \det(\nabla \bm{y}_k)\dfrac{1}{\det(\nabla \bm{y})}\right\rVert_{L^1}\nonumber\\
& &~~~~+\left\lVert\Mvec_k\circ \bm{y}_k \det(\nabla \bm{y}_k)\dfrac{1}{\det(\nabla \bm{y})}-\Mvec\circ \bm{y} \det(\nabla \bm{y})\dfrac{1}{\det(\nabla \bm{y})}\right\rVert_{L^1}\nonumber\\
&\underset{\text{H\"older}}{\leq}&\left\lVert\Mvec_k\circ \bm{y}_k \det(\nabla \bm{y}_k)\right\rVert_{L^s}\left\lVert\dfrac{1}{\det(\nabla \bm{y}_k)}-\dfrac{1}{\det(\nabla \bm{y})}\right\rVert_{L^{p/3}}\nonumber\\
& &~~~+\left\lVert\Mvec_k\circ \bm{y}_k \det(\nabla \bm{y}_k)-\Mvec\circ \bm{y} \det(\nabla \bm{y})\right\rVert_{L^s}\left\lVert\dfrac{1}{\det(\nabla \bm{y})}\right\rVert_{L^{p/3}}.\label{eq:com-est}
\end{eqnarray}
The two terms on the right-hand side of \eqref{eq:com-est} vanish as $k\to\infty$ in virtue of \eqref{eq:con_Lp3} and \eqref{eq:mcom29}. Consequently,
\begin{eqnarray}
\Mvec_k\circ \bm{y}_k\rightarrow\Mvec\circ \bm{y} \mbox{ strongly in } L^1.
\end{eqnarray}
In particular, after passing to a subsequence,
\begin{eqnarray}
\Mvec_k\circ \bm{y}_k\rightarrow\Mvec\circ \bm{y} \mbox{ a.e. in } \Omega.
\end{eqnarray}
\end{proof}
We now discuss the lower semicontinuity aspects of the magnetoelastic energy functional, cf.\ \eqref{eq:en_mag}.
\subsection{Lower semicontinuity}\label{lower_semi}
We next turn to the lower-semicontinuity analysis of the magnetoelastic energy. The main additional difficulty arises from the magnetostatic contribution, since the stray field is determined by the stationary Maxwell equations and therefore depends nonlocally on the magnetization. To pass to the limit in this term, we first establish a convergence result for the associated magnetic field. The following lemma is based on the ideas developed in \cite{DeSimone}; 
see also the concise reformulation given in 
\cite[Lemma 2.3]{MK}.
\begin{lemma}\label{lem-str}
Let $\chi_{\bm{y}_k\left(\Omega\right)}\Mvec_k \rightarrow \chi_{\bm{y}\left(\Omega\right)}\Mvec \mbox{ strongly in } L^{2}\left( \mathbb R^3; \mathbb R^3 \right)$ and $\Hvec_{\Mvec_k}\in L^2\left(\mathbb R^3; \mathbb R^3\right)$ such that
\end{lemma}
%\begin{gather}
\begin{eqnarray}
\textbf{div}\left(\chi_{\bm{y}\left(\Omega\right)}\Mvec + \mu_0\Hvec_{\Mvec}\right)&=0 \mbox{ in } \mathbb R^3,\nonumber\\
\textbf{curl}~\Hvec_{\Mvec}&=0 \mbox{ in } \mathbb R^3,\label{eq:maxwell'}
\end{eqnarray}
%\end{gather}
for each $k\in\mathbb N$. Then $\Hvec_{\Mvec_k} \rightharpoonup \Hvec_{\Mvec}$ weakly in $L^{2}\left( \mathbb R^3; \mathbb R^3 \right)$, where $\Hvec_{\Mvec}$ is a solution of \eqref{eq:maxwell'}.
\begin{proof}
%By virtue of the energy functional $\mathscr{E}_{mag}\left(\bm{y}, \Mvec\right)$, it is inevitable that $\Hvec_{\Mvec}$ is curl-free in the sense of distribution.
In view of the Maxwell system \eqref{eq:maxwell'}, the stray field $\Hvec_{\Mvec}$ naturally belongs to the space 
\begin{eqnarray}
\mathscr{H} = \{\Hvec\in L^2\left(\mathbb R^3; \mathbb R^3\right): \textbf{curl}\Hvec = 0 \mbox{ in the sense of distribution} \}\label{eq:mcom30}
\end{eqnarray}
and $\Hvec_{\Mvec}$ satisfying
\begin{eqnarray}
\mu_0\int_{\mathbb R^3}\Hvec_{\Mvec}\cdot\Hvec dx = -\int_{\mathbb R^3}\chi_{\bm{y}\left(\Omega\right)}\Mvec\cdot\Hvec dx ~~\forall~\Hvec\in \mathscr{H}.\label{eq:mcom31}
\end{eqnarray}
Since $\Hvec$ belongs to $L^2\left(\mathbb R^3; \mathbb R^3\right)$, the Poincar\'e Lemma (see e.g.\ \cite[Theorem 8.3]{csato2011pullback}) implies that 
\begin{eqnarray}
\textbf{curl}\Hvec = 0 \mbox{ in the sense of distribution }
\iff \Hvec = -\nabla \bm{\Phi} \mbox{ for some } \bm{\Phi}\in H^1\left(\mathbb R^3\right).\nonumber\\
\label{eq:mcom32}
\end{eqnarray}
Substituting the representation \eqref{eq:mcom32} into \eqref{eq:mcom31}, we obtain
\begin{eqnarray}
\mu_0\int_{\mathbb R^3}\nabla \bm{\Phi}_{\Mvec}\cdot\nabla\bm{\Phi} dx = \int_{\mathbb R^3}\chi_{\bm{y}\left(\Omega\right)}\Mvec\cdot \nabla\bm{\Phi}dx ~~\forall~\bm{\Phi}\in H^1\left(\mathbb R^3\right),
\end{eqnarray}
where $\Hvec_{\Mvec} = -\nabla \bm{\Phi}_{\Mvec}$ for some $\bm{\Phi}_{\Mvec}\in H^1\left(\mathbb R^3\right)$.\\
\\
Consider the bilinear form $B[\cdot~, \cdot]: \mathbb R^3\times\mathbb R^3\rightarrow\mathbb R$  defined by
\begin{eqnarray}
B[\bm{\Phi}_{\Mvec}, \bm{\Phi}] = \mu_0\int_{\mathbb R^3}\nabla \bm{\Phi}_{\Mvec}\cdot\nabla\bm{\Phi}dx = \int_{\mathbb R^3}\chi_{\bm{y}\left(\Omega\right)}\Mvec\cdot \nabla\bm{\Phi}dx~~\forall~\bm{\Phi}, \bm{\Phi}_{\Mvec}\in H^1\left(\mathbb R^3\right)
\end{eqnarray}
such that
\begin{eqnarray}
|B[\bm{\Phi}_{\Mvec}, \bm{\Phi}]|\le \mu_0||\bm{\Phi}_{\Mvec}||_{H^1}||\bm{\Phi}||_{H^1}~~\forall~\bm{\Phi}_{\Mvec}, \bm{\Phi}\in H^1\left(\mathbb R^3\right)
\end{eqnarray}
and
\begin{eqnarray}
B[\bm{\Phi}_{\Mvec}, \bm{\Phi}_{\Mvec}]\ge \mu_0||\bm{\Phi}_{\Mvec}||^2.
\end{eqnarray}
The continuity and coercivity properties of the bilinear form allow us to apply the Lax-Milgram theorem  \cite[Section 6.2,Theorem 1]{evans2022partial}. Consequently, there exists a unique potential $\bm{\Phi}_{\Mvec}\in H^1\left(\mathbb R^3\right)$ satisfying
\begin{eqnarray}
B[\bm{\Phi}_{\Mvec}, \bm{\Phi}] = \mu_0\int_{\mathbb R^3}\nabla \bm{\Phi}_{\Mvec}\cdot\nabla\bm{\Phi} dx = \int_{\mathbb R^3}\chi_{\bm{y}\left(\Omega\right)}\Mvec\cdot \nabla\bm{\Phi}dx~~\forall~\bm{\Phi}\in H^1\left(\mathbb R^3\right).
\end{eqnarray}
It follows that the Maxwell system \eqref{eq:maxwell'} admits a unique weak solution $\Hvec_{\Mvec}\in L^2\left(\mathbb R^3; \mathbb R^3\right)$ for every magnetization $\Mvec\in L^2\left(\mathbb R^3; \mathbb R^3\right)$.\\
\\
The uniform energy estimate \eqref{eq:en_bound} yields 
\begin{eqnarray}
\sup_{k\in \mathbb N}\int_{\mathbb R^3}|\Hvec_{\Mvec_k}|^2dx<\infty.
\end{eqnarray}
Applying the Banach-Alaoglu theorem, we obtain a subsequence (without re-indexing) such that
\begin{eqnarray}
\Hvec_{\Mvec_k}\rightharpoonup \bm{\eta} \mbox{ weakly in } L^2\left(\mathbb R^3; \mathbb R^3\right).
\end{eqnarray}
Finally, uniqueness of solutions to \eqref{eq:maxwell'}, together with the linear dependence of the stray field on the magnetization, identifies the weak limit $\bm{\eta}$ with $\Hvec_{\Mvec}$. This completes the proof.
\end{proof}
We now conclude the lower semicontinuity argument showing the following theorem.
\begin{theorem}\label{lower-semi}
Suppose that the admissible set $\mathscr{A}$ is non-empty. Then,
\begin{eqnarray}
\mathscr{E}_{mag}\left(\bm{y}, \Mvec\right)\le\liminf_{k\rightarrow\infty}\mathscr{E}_{mag} \left(\bm{y}_k, \Mvec_k\right),
\mbox{ when } \left(\bm{y}_k, \Mvec_k\right)\rightarrow\left(\bm{y}, \Mvec\right) \mbox{ in } \mathscr{A}.
\end{eqnarray}
\end{theorem}
\begin{proof}
We verify lower semicontinuity separately for each contribution to the magnetoelastic energy and then combine the resulting estimates.\newline\\
\textbf{Step I.} Lower semicontinuity of the magnetostrictive energy:\newline\\
\noindent The weak convergence of the deformation implies the standard weak continuity properties of the minors. More precisely, \cite[Theorem 8.20]{dacorogna2007direct} yields
\begin{eqnarray}
\mbox{Cof}\nabla \bm{y}_k &\rightharpoonup& \mbox{Cof}\nabla \bm{y} \mbox{ weakly in } L^{p/2}\left(\Omega; \mathbb R^{3\times 3}\right),\\
\det(\nabla\bm{y}_k) &\rightharpoonup&\det(\nabla\bm{y}) \mbox{ weakly in } L^{p/3}\left(\Omega; \mathbb R\right).
\end{eqnarray} 
Furthermore, \Cref{saturation} provides a subsequence for which
\begin{eqnarray}
\Mvec_k \circ \bm{y}_k \rightarrow \Mvec\circ \bm{y} \mbox{ a.e. in } \Omega.
\end{eqnarray}
Using the polyconvexity property of the energy density introduced in \eqref{eq:poly}, we obtain
\begin{eqnarray}\label{semc1}
& &\int_{\Omega}\mathcal{W}\left(\nabla \bm{y}\left(X\right), \Mvec\left(\bm{y}\left(X\right)\right)\right)dX\nonumber\newline\\
&\underset{\eqref{eq:poly}}{=}&\int_{\Omega}\overline{\mathcal{W}}\left(\nabla\bm{y}, \mathbf{Cof}\left(\nabla\bm{y}\right), \det(\nabla\bm{y}), \Mvec\left(\bm{y}\left(X\right)\right)\right)dX\nonumber\newline\\
&\le&\liminf_{k\rightarrow\infty}\int_{\Omega}\overline{\mathcal{W}}\left(\nabla \bm{y}_k\left(X\right), \mbox{Cof}\nabla \bm{y}_k, \det(\nabla\bm{y}_k), \Mvec_k\left(\bm{y}_k\left(X\right)\right)\right)dX\nonumber\newline\\
&\underset{\eqref{eq:poly}}{=}&\liminf_{k\rightarrow\infty}\int_{\Omega}\mathcal{W}\left(\nabla \bm{y}_k\left(X\right), \Mvec_k\left(\bm{y}_k\left(X\right)\right)\right)dX.\label{eq:poly-low}
\end{eqnarray}
The second inequality follows from the lower-semicontinuity theorem for polyconvex integrands; see  \cite[Theorem 5.4]{ball1981null}.\newline\\
\textbf{Step II.} Lower semicontinuity of the exchange energy:\newline\\
Using absolute continuity of the Lebesgue integral, for every $\epsilon>0$, one can choose $\delta>0$ such that
\begin{eqnarray}
\int_{O_{\delta}\left(\bm{y}\right)}|\nabla \Mvec|^2dx\ge\int_{\bm{y}\left(\Omega\right)}|\nabla \Mvec|^2dx-\epsilon.\label{eq:wklw}
\end{eqnarray}
The weak convergence established in \eqref{eq:mcom7}, combined with the lower semicontinuity of the $L^2$-norm, gives 
\begin{eqnarray}
\liminf_{k\to\infty} \int_{\bm{y}_k\left(\Omega\right)}|\nabla \Mvec_k|^2dx&\ge&\liminf_{k\to\infty} \int_{O_{\delta}\left(\bm{y}\right)}|\nabla \Mvec_k|^2dx\nonumber\newline\\
&\ge&\int_{O_{\delta}\left(\bm{y}\right)}|\nabla \Mvec|^2dx\nonumber\newline\\
&\underset{\eqref{eq:wklw}}{\ge}&\int_{\bm{y}\left(\Omega\right)}|\nabla \Mvec|^2dx-\epsilon\label{semc2}
\end{eqnarray}
for any $\epsilon>0$.\newline\\
%\noindent Due to the lower semicontinuity property of the norm, it is concluded that, $\chi_{\Omega^{y_k}}\nabla \Mvec_k \rightharpoonup \chi_{\Omega^{\bm{y}}}\nabla\Mvec \mbox{ in } L^2\left(\mathbb R^3; \mathbb R^{3\times 3}\right)$ implies
%\begin{eqnarray}\label{semc2}
%\int_{\Omega^{\bm{y}}}|\nabla \Mvec|^2dx \le \liminf_{k\to\infty} \int_{\Omega^{\bm{y}_k}}|\nabla \Mvec_k|^2dx.
%\end{eqnarray}
\textbf{Step III.} Lower semicontinuity of the stray field energy:\newline\\
\noindent Applying Lemma \ref{lem-str}, together with weak lower semicontinuity of the $L^2$-norm, we obtain
\begin{eqnarray}\label{semc3}
\int_{\mathbb R^3}|\Hvec_{\Mvec}|^2dx\le\liminf_{k\rightarrow\infty} \int_{\mathbb R^3} |\Hvec_{\Mvec_k}|^2dx.
\end{eqnarray}
Combining the conclusions of Steps I-III, namely \eqref{semc1}, \eqref{semc2} and \eqref{semc3} establishes the desired lower-semicontinuity inequality and completes the proof.
\end{proof}
\subsection{Proof of the existence theorem}\label{Proof of existence theorem}
\textbf{Proof of \Cref{thm1}}
\begin{proof}
The result follows from the direct method of the calculus of variations; see, for example, \cite[Theorem 2.1]{Rindler}. Indeed, \Cref{compact} provides compactness of energy-bounded sequences in the admissible class, while \Cref{lower-semi} yields lower semicontinuity of the magnetoelastic energy functional with respect to the convergence introduced above. Consequently, every minimizing sequence admits an admissible limit that attains the minimum of the energy.

It follows by the direct method in the calculus of variations (e.g.\ \cite[Theorem 2.1]{Rindler}) using \Cref{compact} and \Cref{lower-semi}.
\end{proof}
%\begin{remark}\label{re7}
%It is important to notice that the definition of curl in the sense of distribution inevitably does not coincide with the curl as a weak sense. However, one can achieve a better regularity of the stray field $\Hvec_{\Mvec}$ in the sense that $\Hvec_{\Mvec}\in H^1_{loc}\left(\mathbb R^3; \mathbb R^3\right)$ by using the elliptic regularity theorem so that the distributional curl coincides with the curl in the weak sense.    
%\end{remark}
%\medskip
\begin{remark}\label{re8}
The model can be extended in a straightforward manner to include the interaction with a prescribed external magnetic field. The corresponding Zeeman contribution is given by
$$-\int_{\mathbb R^3}\chi_{\bm{y}\left(\Omega\right)}\Mvec\cdot\Hvec_{ext}dx,$$ 
where $\Hvec_{ext}$ denotes the applied magnetic field. Since this field is regarded as prescribed data, Young's inequality immediately provides a suitable lower bound for the additional energy term. The compactness argument remains unchanged because it relies on the convergence properties of the magnetization, and weak lower semicontinuity of the Zeeman contribution follows directly from Fatou's lemma. 
\end{remark}
\section{Acknowledgements}
S.D. is grateful to DAAD (German Academic Exchange Service): The work is part of the PhD thesis of S.D.\ and was made possible with the support of a scholarship from the German Academic Exchange Service (DAAD). Also, S.D. is grateful to Humboldt foundation for the support of her research through Humboldt Professorship award of Stefanie Petermichl. The authors thank Anna Doležalová for carefully reviewing some part of the work.
\bibliographystyle{alpha}  % Numbered, citation-order bibliography
\bibliography{sn-bibliography}      % Replace 'references' with your .bib file name

@article{Barchiesi,
  title={Local invertibility in Sobolev spaces with applications to nematic elastomers and magnetoelasticity},
  author={Barchiesi, Marco and Henao, Duvan and Mora-Corral, Carlos},
  journal={Archive for Rational Mechanics and Analysis},
  volume={224},
  pages={743--816},
  year={2017},
  publisher={Springer}
}

@article{dolevzalova2025global,
  title={Global Integrability of the Reciprocal of Jacobians for Homeomorphisms of Finite Distortion},
  author={Dole{\v{z}}alov{\'a}, Anna and Onninen, Jani and Zhu, Yizhe and Zhu, Zheng},
  journal={arXiv preprint arXiv:2512.15519},
  year={2025}
}

@article{dolevzalova2025inv,
  title={(INV) condition and regularity of the inverse},
  author={Dole{\v{z}}alov{\'a}, Anna and Hencl, Stanislav and Onninen, Jani},
  journal={Journal of Functional Analysis},
  pages={111215},
  year={2025},
  publisher={Elsevier}
}

@article{Bresciani22,
  title={Quasistatic evolution in magnetoelasticity under subcritical coercivity assumptions},
  author={Bresciani, Marco},
  journal={Calculus of Variations and Partial Differential Equations},
  volume={62},
  number={6},
  pages={181},
  year={2023},
  publisher={Springer}
}

@article{Bresciani21,
  title={Existence results in large-strain magnetoelasticity},
  author={Bresciani, Marco and Davoli, Elisa and Kru{\v{z}}{\'\i}k, Martin},
  journal={Annales de l'Institut Henri Poincar{\'e} C},
  volume={40},
  number={3},
  pages={557--592},
  year={2022}
}

@article{DeSimone,
  title={Energy minimizers for large ferromagnetic bodies},
  author={De Simone, Antonio},
  journal={Archive for Rational Mechanics and Analysis},
  volume={125},
  pages={99--143},
  year={1993},
  publisher={Springer-Verlag}
}

@article{Desimmone,
  title={Existence of Minimizers for a Variational Problem in Two-Dimensional Nonlinear Magnetoelasticity},
  author={DeSimone, Antonio and Dolzmann, Georg},
  journal={Archive for Rational Mechanics and Analysis},
  volume={144},
  pages={107--120},
  year={1998},
  publisher={Springer}
}

@incollection{Evans-Gariepy,
  author={Evans, Lawrence C. and Ronald, Gariepy F.},
  title={Measure theory and fine properties of functions},
  year={2018},
  publisher={Routledge}
}

@incollection{hubert2008magnetic,
  title={Magnetic domains: the analysis of magnetic microstructures},
  author={Hubert, Alex and Sch{\"a}fer, Rudolf},
  year={2008},
  publisher={Springer Science \& Business Media}
}

@incollection{Fonseca-Leoni,
  author={Fonseca, Irene and Leoni, Giovanni},
  title={Modern methods in the calculus of variations: $\mathrm{L}^p$ spaces},
  year={2007},
  publisher={Springer Science \& Business Media}
}

@article{yan2000weak,
  title={On the weak limit of mappings with finite distortion},
  author={Yan, Baisheng},
  journal={Proceedings of the American Mathematical Society},
  volume={128},
  number={11},
  pages={3335--3340},
  year={2000}
}

@book{csato2011pullback,
  title={The pullback equation for differential forms},
  author={Csat{\'o}, Gyula and Dacorogna, Bernard and Kneuss, Olivier},
  volume={83},
  year={2011},
  publisher={Springer Science \& Business Media}
}

@article{Gehring1999,
  title={The limit of mappings with finite distortion},
  author={Gehring, F.W. and Iwaniec, T.},
  journal={Annales Fennici Mathematici},
  volume={24},
  number={1},
  pages={253--264},
  year={1999}
}

@article{Grandi2019,
  title={A phase-field approach to Eulerian interfacial energies},
  author={Grandi, Diego and Kru{\v{z}}{\'\i}k, Martin and Mainini, Edoardo and Stefanelli, Ulisse},
  journal={Archive for Rational Mechanics and Analysis},
  volume={234},
  number={1},
  pages={351--373},
  year={2019},
  publisher={Springer}
}

@incollection{Hencl-Koskela,
  title={Lectures on mappings of finite distortion},
  author={Hencl, Stanislav and Koskela, Pekka},
  volume={2096},
  year={2014},
  publisher={Springer}
}

@article{Iwaniec_Sverak,
  title={On mappings with integrable dilatation},
  author={Iwaniec, Tadeusz and Sverak, Vladimir},
  journal={Proceedings of the American Mathematical Society},
  pages={181--188},
  year={1993},
  publisher={JSTOR}
}

@article{James,
  title={Theory of magnetostriction with applications to TbxDy1-xFe2},
  author={James, Richard D and Kinderlehrer, David},
  journal={Philosophical Magazine B},
  volume={68},
  number={2},
  pages={237--274},
  year={1993},
  publisher={Taylor \& Francis}
}

@article{MK,
  title={Existence results for incompressible magnetoelasticity},
  author={Kru{\v{z}}{\'\i}k, Martin and Stefanelli, Ulisse and Zeman, Jan},
  journal={Discrete and Continuous Dynamical Systems},
  volume={35},
  number={6},
  pages={2615--2623},
  year={2014},
  publisher={Discrete and Continuous Dynamical Systems}
}

@article{Molchanova,
  title={Injectivity almost everywhere and mappings with finite distortion in nonlinear elasticity},
  author={Molchanova, Anastasia and Vodopyanov, Sergey},
  journal={Calculus of Variations and Partial Differential Equations},
  volume={59},
  number={1},
  pages={17},
  year={2020},
  publisher={Springer}
}

@article{Onninen,
  title={Regularity of the inverse of spatial mappings with finite distortion.},
  author={Onninen, Jani},
  journal={Calculus of Variations \& Partial Differential Equations},
  volume={26},
  number={3},
  year={2006}
}

@incollection{Rindler,
  author		= "Rindler, F.",
  title			= "Calculus of Variations",
  publisher		= "Springer",
  year			= "2018"
}

@article{Luskin,
  title={Existence of energy minimizers for magnetostrictive materials},
  author={Rybka, Piotr and Luskin, Mitchell},
  journal={SIAM Journal on Mathematical Analysis},
  volume={36},
  number={6},
  pages={2004--2019},
  year={2005},
  publisher={SIAM}
}

@incollection{Vaisala,
  author="V{\"a}is{\"a}l{\"a}, Jussi",
  title="Lectures on n-dimensional quasiconformal mappings",
  publisher="Springer",
  year="2006"
}

@incollection{brown1966magnetoelastic,
  title="Magnetoelastic interactions",
  author="Brown, William Fuller",
  publisher="Springer",
  year="1966"
}

@article{Nature,
  title={Soft fibers with magnetoelasticity for wearable electronics},
  author={Zhao, Xun and Zhou, Yihao and Xu, Jing and Chen, Guorui and Fang, Yunsheng and Tat, Trinny and Xiao, Xiao and Song, Yang and Li, Song and Chen, Jun},
  journal={Nature Communications},
  volume={12},
  number={1},
  pages={6755},
  year={2021},
  publisher={Nature Publishing Group UK London}
}

@article{rogers1986steady,
  title={Steady-state problems of nonlinear electro-magneto-thermo-elasticity},
  author={Rogers, Robert C and Antman, Stuart S},
  journal={Technical Summary Report Wisconsin Univ},
  year={1986}
}

@article{rogers1993existence,
  title={Existence results for large deformations of magnetostrictive materials},
  author={Rogers, Robert C},
  journal={Journal of Intelligent Material Systems and Structures},
  volume={4},
  number={4},
  pages={477--483},
  year={1993},
  publisher={Sage Publications Sage CA: Thousand Oaks, CA}
}

@article{rogers1988nonlocal,
  title={Nonlocal variational problems in nonlinear electromagneto-elastostatics},
  author={Rogers, Robert C},
  journal={SIAM Journal on Mathematical Analysis},
  volume={19},
  number={6},
  pages={1329--1347},
  year={1988},
  publisher={SIAM}
}

@article{ball1976convexity,
  title={Convexity conditions and existence theorems in nonlinear elasticity},
  author={Ball, John M},
  journal={Archive for Rational Mechanics and Analysis},
  volume={63},
  pages={337--403},
  year={1976},
  publisher={Springer}
}

@article{ball1981null,
  title={Null Lagrangians, weak continuity, and variational problems of arbitrary order},
  author={Ball, John M and Currie, John C and Olver, Peter J},
  journal={Journal of Functional Analysis},
  volume={41},
  number={2},
  pages={135--174},
  year={1981},
  publisher={Elsevier}
}

@book{ciarlet2021mathematical,
  title={Mathematical elasticity: Three-dimensional elasticity},
  author={Ciarlet, Philippe G},
  year={2021},
  publisher={SIAM}
}

@book{dacorogna2007direct,
  title={Direct methods in the calculus of variations},
  author={Dacorogna, Bernard},
  volume={78},
  year={2007},
  publisher={Springer Science \& Business Media}
}

@article{fonseca1995local,
  title={Local invertibility of Sobolev functions},
  author={Fonseca, Irene and Gangbo, Wilfrid},
  journal={SIAM Journal on Mathematical Analysis},
  volume={26},
  number={2},
  pages={280--304},
  year={1995},
  publisher={SIAM}
}

@phdthesis{Dutta-thesis,
    author = {Dutta, Shilpa},
    title = {Existence of minimizers in magnetoelasticity and a ferronematic model with a stray field energy},
    school = {University of Würzburg},
    year = 2026,
    url = {https://nbn-resolving.org/urn:nbn:de:bvb:20-opus-443335}
}

@book{evans2022partial,
  title={Partial differential equations},
  author={Evans, Lawrence C.},
  volume={19},
  year={2022},
  publisher={American Mathematical Society}
}

\end{document}